\documentclass{amsart}
\usepackage{amsthm,mathtools}
\usepackage[all]{xy}

\theoremstyle{plain}
\theoremstyle{definition}
\newtheorem{theorem}{Theorem}[section]
\newtheorem{lemma}[theorem]{Lemma}

\newtheorem{problem}[theorem]{Problem}
\newtheorem{question}[theorem]{Question}
\newtheorem{proposition}[theorem]{Proposition}
\newtheorem{corollary}[theorem]{Corollary}

\newtheorem{definition}[theorem]{Definition}
\newtheorem{example}[theorem]{Example}
\newtheorem{remark}[theorem]{Remark}

\usepackage{amssymb,amsmath,tabularx,graphicx,float,placeins, hyperref}
\usepackage{tikz}
\usetikzlibrary[calc,intersections,through,backgrounds,arrows,decorations.pathmorphing]

\def\Sfrak{\mathfrak{S}}

\def\Bcal{\mathcal{B}}\def\Gcal{\mathcal{G}}\def\Jcal{\mathcal{J}}\def\Ical{\mathcal{I}}\def\Lcal{\mathcal{L}}\def\Mcal{\mathcal{M}}\def\Ncal{\mathcal{N}}\def\Wcal{\mathcal{W}}\def\Xcal{\mathcal{X}}\def\Ycal{\mathcal{Y}}\def\Zcal{\mathcal{Z}}

\def\Fbb{\mathbb{F}}\def\Kbb{\mathbb{K}}\def\Nbb{\mathbb{N}}\def\Pbb{\mathbb{P}}\def\Rbb{\mathbb{R}}\def\Zbb{\mathbb{Z}}

\def\hookra{\hookrightarrow}\def\LRa{\Leftrightarrow}\def\longlra{\longleftrightarrow}\def\ov{\overline}\def\pr{\prime}\def\ra{\rightarrow}\def\setm{\setminus}\def\thra{\twoheadrightarrow}

\DeclareMathOperator{\Comp}{Comp}  
\DeclareMathOperator{\Con}{Con}  
\DeclareMathOperator{\con}{con}  
\DeclareMathOperator{\des}{des}  
\DeclareMathOperator{\Int}{Int}  
\DeclareMathOperator{\MTub}{MTub}  
\DeclareMathOperator{\NCSym}{NCSym}  
\DeclareMathOperator{\std}{std}  

\DeclareMathOperator{\topT}{top} 

\newcommand{\down}{\downarrow}

\newcommand{\covers}{{\,\,\,\cdot\!\!\!\! >\,\,}}
\newcommand{\covered}{{\,\,<\!\!\!\!\cdot\,\,\,}}
\renewcommand{\top}{\operatorname{top}}
\newcommand{\pidown}{\pi_\downarrow}

\newcommand{\inv}{\operatorname{inv}}
\newcommand{\meet}{\wedge}
\newcommand{\join}{\vee}

\newcounter{margincounter}
\begin{document}

\title{Lattices from graph associahedra and subalgebras of the Malvenuto-Reutenauer algebra}

\author{Emily Barnard, Thomas McConville}

\maketitle

\begin{abstract}
The Malvenuto-Reutenauer algebra is a well-studied combinatorial Hopf algebra with a basis indexed by permutations. This algebra contains a wide variety of interesting sub Hopf algebras, in particular the Hopf algebra of plane binary trees introduced by Loday and Ronco. We compare two general constructions of subalgebras of the Malvenuto-Reutenauer algebra, both of which include the Loday-Ronco algebra. The first is a construction by Reading defined in terms of lattice quotients of the weak order, and the second is a construction by Ronco in terms of graph associahedra. To make this comparison, we consider a natural partial ordering on the maximal tubings of a graph and characterize those graphs for which this poset is a lattice quotient of the weak order.
\end{abstract}

\setcounter{tocdepth}{2}
\tableofcontents

\section{Introduction}

Given a graph $G$, Postnikov defined a graph associahedron $P_G$ as an example of a \emph{generalized permutohedron}, a polytope whose normal fan coarsens the braid arrangement \cite{postnikov:2009permutohedra}. Graph associahedra were also introduced independently in \cite{carr.devadoss:2006coxeter} and \cite{davis.janus.scott:2003fundamental}. Some significant examples of graph associahedra include the associahedron, the cyclohedron, and the permutohedron. Combinatorially, the faces of the graph associahedron correspond to certain collections of connected subgraphs of $G$, called \emph{tubings}. We recall these definitions in Section~\ref{sec:tubing}. We consider a poset $L_G$ on the maximal tubings of $G$ whose Hasse diagram is an orientation of the $1$-skeleton of the graph associahedron.


In \cite{ronco:2012tamari}, Ronco defined a binary operation on a vector space generated by the tubings of an ``admissible'' family of graphs $\Gcal$, which gives this space the structure of an associative algebra. We call this algebra a \emph{tubing algebra}; see Section~\ref{subsec_hopf_algebra}. In particular, when $\Gcal$ is the set of complete graphs $K_n$ or path graphs $P_n$, the tubing algebra is isomorphic to either the Malvenuto-Reutenauer algebra on permutations \cite{malvenuto.reutenauer:1995duality} or the Loday-Ronco algebra on binary trees \cite{loday.ronco:1998hopf}, respectively. The interpretation of these algebras in terms of tubings was given previously in \cite{forcey.springfield:2010geometric}.

In Section~\ref{subsec_tubing_coalgebra}, we introduce the notion of a ``restriction-compatible'' family of graphs. Such families come with a comultiplication on their maximal tubings. We call the resulting coalgebra a \emph{tubing coalgebra}.

Reading introduced a general technique to construct subalgebras of the Malvenuto-Reutenauer algebra using lattice quotients of the weak order on permutations in \cite{reading:2005lattice}. Using the terminology of \cite{reading:2005lattice}, if a sequence of lattice congruences $\{\Theta_n\}_{n\geq 0}$ is translational (respectively, insertional), then the set of congruence classes of $\Sfrak_n$ modulo $\Theta_n$ naturally index a basis of a subalgebra (respectively, sub-coalgebra) of the Malvenuto-Reutenauer algebra.


The main goal of this work is to compare the above constructions of Reading and Ronco. For any graph $G$ with vertex set $[n]$, there is a canonical surjective map ${\Psi_G:\Sfrak_n\ra L_G}$ obtained by coarsening the braid arrangement in $\Rbb^n$ to the normal fan of~$P_G$. Our first main result characterizes graphs for which the map $\Psi_G$ is a lattice map. We say a graph $G$ is \emph{filled} if for each edge $\{i,k\}$ in $G$, there are edges $\{i,j\}$ and $\{j,k\}$ in $G$ whenever $i<j<k$.

\begin{theorem}\label{thm_main_lattice}
The map $\Psi_G$ is a lattice quotient map if and only if $G$ is filled.
\end{theorem}

Restricting attention to filled graphs, we have the following comparison between the constructions of Reading and Ronco.

\begin{theorem}\label{thm_main}
  Let $\Gcal=\{G_n\}_{n\geq 0}$ be a sequence of filled graphs, and let $\mathbf{\Theta}=\{\Theta_n\}_{n\geq 0}$ be the associated sequence of lattice congruences of the weak order.
  \begin{enumerate}
  \item\label{thm_main_1} The family $\Gcal$ is admissible if and only if $\mathbf{\Theta}$ is translational.
  \item\label{thm_main_2} The family $\Gcal$ is restriction-compatible if and only if $\mathbf{\Theta}$ is insertional.
  \end{enumerate}
\end{theorem}



In \cite{forcey:2012species}, Forcey posed the problem of determining whether $L_G$ is a lattice for any graph $G$. This turns out to be false in general; cf. Section~\ref{subsec:tubing_lattice}. We say a graph $G$ on $[n]$ is \emph{right-filled} if whenever $\{i,k\}$ is an edge, so is $\{j,k\}$ for $i<j<k$. Dually, we say $G$ is \emph{left-filled} if $\{i,j\}$ is an edge whenever there is an edge $\{i,k\}$ for $i<j<k$. We prove that $L_G$ is a lattice whenever $G$ is either left-filled or right-filled. More precisely, these are the cases when $L_G$ is a semilattice quotient of the weak order. For other graphs, the poset $L_G$ may still be a lattice, even if it is not a semilattice quotient of the weak order. Some additional examples and conjectures are discussed in Section~\ref{sec:other}. 


The rest of the paper is organized as follows. We introduce the poset of maximal tubings $L_G$ in Section~\ref{sec:tubing}. The main result in this section is Theorem~\ref{thm:NRC}, which states that $L_G$ has the \emph{non-revisiting chain property}, defined in Section~\ref{subsec:NRC}. In Section~\ref{sec:lattice}, we recall the congruence-uniform lattice structure of the weak order on permutations and elaborate on the canonical map from permutations to maximal tubings. Sections~\ref{sec_lattice} and~\ref{sec_hopf} are devoted to proving Theorems~\ref{thm_main_lattice} and~\ref{thm_main}, respectively. We end the paper with some open problems and conjectures in Section~\ref{sec:other}.

\section{Poset of maximal tubings}\label{sec:tubing}

\subsection{Tubings and $G$-trees}

In this section, we recall the principal combinatorial objects in this paper, namely the maximal tubings of a graph and $G$-trees.

Let $G=(V,E)$ be a simple graph with vertex set $V=[n]:=\{1,\ldots,n\}$. If $I\subseteq V$, we let $G|_I$ denote the induced subgraph of $G$ with vertex set $I$. A \emph{tube} is a nonempty subset $I$ of vertices such that the induced subgraph $G|_I$ is connected. Any tube not equal to $V$ is called a \emph{proper tube}. We let $\Ical(G)$ be the set of all tubes of $G$.

We define the \emph{deletion} $G\setm I$ to be the graph $G|_{V\setm I}$ and the \emph{contraction} (or \emph{reconnected complement}) $G/I$ as the graph with vertex set $V\setm I$ and edges $\{i,j\},\ (i\neq j)$ if either $\{i,j\}\in E(G)$ or there exists a tube $J$ of $G|_I$ such that $\{i,k\}\in E(G)$ and $\{j,l\}\in E(G)$ for some $k,l\in J$.
 Note that we define deletion and contraction on sets of vertices of $G$ rather than on edges as it is done for graphic matroids. Furthermore, the contracted graph $G/I$ is always simple, i.e. it has no loops or parallel edges.

Two tubes $I, J$ are said to be \emph{compatible} if either

\begin{itemize}
\item they are \emph{nested}: $I\subseteq J$ or $J\subseteq I$, or
\item they are \emph{separated}: $I\cup J$ is not a tube.
\end{itemize}

A \emph{tubing} $\Xcal$ of $G$ is any collection of pairwise compatible tubes. The collection $\Xcal$ is said to be a \emph{maximal tubing} if it is maximal by inclusion. We let $\MTub(G)$ be the set of maximal tubings of the graph $G$. If $\Xcal$ is a tubing of $G$ and $X_1,\ldots,X_r\in\Xcal$ are pairwise disjoint, then the union $I=X_1\cup\cdots\cup X_r$ is called an \emph{ideal} of $\Xcal$. This terminology may be explained by the connection to $G$-trees given later in this section.

\begin{lemma}
If $\Xcal$ is a tubing of $G$ with an ideal $I$ then there is a unique collection $X_1,\ldots, X_r$ of pairwise disjoint tubes in $\Xcal$, namely the connected components of $G|_I$, such that $I=X_1\cup\cdots\cup X_r$.
\end{lemma}

Tubings of $G$ may be restricted to certain induced subgraphs or contracted graphs as follows.
If $I$ is a subset of $[n]$, let $\Comp(I)$ be the set of maximal tubes of $G|_I$; i.e., $J\in\Comp(I)$ if $J\subseteq I$ and $G|_J$ is a connected component of $G|_I$. If $\Xcal$ is a tubing of $G$, we set
$$\Xcal|_I:=\bigcup_{J\in\Xcal}\Comp(I\cap J).$$

\begin{lemma}\label{lem:tubing_restriction}
Let $\Xcal$ is a tubing of $G$ and $I\subseteq [n]$.
  The collection $\Xcal|_I$ is a tubing of $G|_I$. If $\Xcal$ is maximal then so is $\Xcal|_I$.
\end{lemma}

Lemma~\ref{lem:tubing_restriction} can be deduced from a cellular map between different graph associahedra; see \cite[Definition 3.4]{forcey.springfield:2010geometric}. This map is a generalized form of the \emph{Tonks projection}, one of the standard maps from the faces of the permutahedron to the faces of the associahedron.

When $I$ is an ideal of $\Xcal$ we set
$$\Xcal/I:=\{J\setm I:\ J\in\Xcal,\ J\nsubseteq I\}.$$
\begin{lemma}\label{lem:tubing_del_con}
Let $\Xcal$ is a tubing of $G$ with an ideal $I$.
The collection $\Xcal/I$ is a tubing of~$G/I$.
If $\Xcal$ is maximal then so is $\Xcal/I$.
\end{lemma}

Any maximal tubing $\Xcal$ contains exactly $n$ tubes. Indeed, we have the following bijection between $\Xcal$ and $[n]$.

\begin{lemma}
  If $\Xcal$ is a maximal tubing, then each tube $I$ contains a unique element $\topT_{\Xcal}(I)\in [n]$ not contained in any proper tube of $\Xcal|_I$. Furthermore, the function $\topT_{\Xcal}$ is a bijection between the tubes in $\Xcal$ and the vertex set $[n]$.
\end{lemma}
\begin{proof}
It is straight forward to check that $\topT_\Xcal(I)$ is well-defined for each tube $I\in \Xcal$.
Let $k\in [n]$ and let $\Ical$ be the set of tubes in $\Xcal$ which contain $k$.
Observe that $\Ical$ is not empty (because the connected component of $G$ containing $k$ is a tube in $\Xcal$.)
Because each of the tubes in $\Ical$ are nested, there is a smallest tube $I\in \Ical$ (under containment) which contains $k$.
For this tube, we have $\topT_\Xcal(I) =k$.

It follows that if $\topT_\Xcal(I)=\topT_\Xcal(J)=k$ then $I=J$.
Therefore $\topT_\Xcal$ is indeed a bijection.
\end{proof}

\begin{figure}
  
  \centering
  \includegraphics{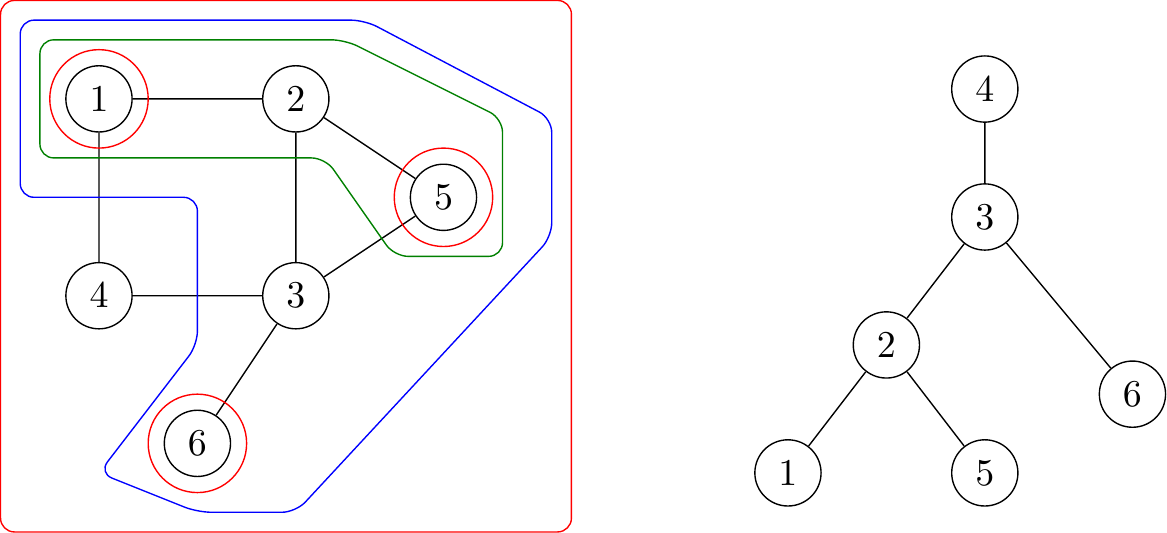}
  \caption{\label{fig:tubing_tree}(left) A maximal tubing (right) Its associated $G$-tree}
  
\end{figure}

Let $T$ be a forest poset on $[n]$.
That is, each connected component of $T$ is a rooted tree, and $i<_T k$ whenever $i$ and $k$ belong to the same connected component, and the unique path from $i$ to the root of this component passes through $k$.
%
Let $i_\down$ denote the principal order ideal generated by $i$ in $T$.
We say that $T$ is a \emph{$G$-forest}, or \emph{$G$-tree} when $T$ is connected, if it satisfies both of the following conditions (see also \cite[Definition~8.1]{postnikov.reiner.williams:2008faces}):
\begin{itemize}
\item For each $i\in [n]$, the set $i_\down$ is a tube of $G$;
\item If $i$ and $k$ are incomparable in $T$, then $i_\down \cup k_\down$ is not tube of $G$.
\end{itemize}

Given a $G$-forest $T$, observe that the collection $\chi(T)=\{i_\down: i\in [n]\}$ is a maximal tubing on $G$.
Indeed, consider $I=i_\down$ and $J=k_\down$ for any pair $i$ and $k$ in $[n]$.
If $i$ and $k$ are not comparable, then it is immediate that $I$ and $J$ are compatible (because $I\cup J$ is not a tube).
On the other hand, if $i$ and $k$ are comparable, then either $I\subset J$ or $J\subset I$.
The following theorem is essentially \cite[Proposition~8.2]{postnikov.reiner.williams:2008faces}, specialized to the case where the building set $\Bcal$ is the collection of tubes of~$G$. 
An example of this correspondence is shown in Figure~\ref{fig:tubing_tree}.
\begin{theorem}\label{G-trees}
Let $G$ be a graph with vertex set $[n]$.
Then the map $\chi$ which sends $T\mapsto \{i_\down: i\in [n]\}$ is a bijection from the set of $G$-forests to the set of maximal tubings on $G$.
The inverse to $\chi$, which we denote by $\tau$ maps the maximal tubing $\Xcal$ to a tree-poset $T$ satisfying: $\top_\Xcal(I)<\top_\Xcal(J)$ if and only if $I\subset J$, where $I$ and $J$ are tubes in $\Xcal$.
\end{theorem}

It follows that $G$ is connected if and only if each $G$-forest is actually a $G$-tree. 


\subsection{Graph associahedra}\label{subsec:graph_assoc}

Before defining the graph associahedron, the main polytopes discussed in this paper, we recall the definition of the normal fan of a polytope.

A \emph{(polyhedral) fan} $\Ncal$ is a set of cones in $\Rbb^n$ such that for any two elements $C,C^{\pr}\in\Ncal$, their intersection $C\cap C^{\pr}$ is in $\Ncal$ and it is a face of both $C$ and $C^{\pr}$. It is \emph{complete} if $\bigcup_{C\in\Ncal} C=\Rbb^n$ and \emph{pointed} if $\{0\}\in\Ncal$. A pointed fan $\Ncal$ is \emph{simplicial} if the number of extreme rays of each $C\in\Ncal$ is equal to its dimension. We consider a simplicial fan to be a type of ``realization'' of a simplicial complex; more accurately, it is a cone over a geometric realization.

For a polytope $P\subseteq\Rbb^n$ and $f\in(\Rbb^n)^*$ in the dual space, we let $P^f$ be the subset of $P$ at which $f$ achieves its maximum value. We consider an equivalence relation on $(\Rbb^n)^*$ where $f\sim g$ if $P^f=P^g$. It is not hard to show that each equivalence class is a relatively open polyhedral cone. The \emph{normal fan} of $P$ is the set of closures of these cones, which forms a complete polyhedral fan. A polytope is simple if and only if its normal fan is simplicial.

The set of tubings of a graph forms a flag simplicial complex $\Delta_G$, called the \emph{nested set complex}. A set $W$ consisting of the vertices of a connected component of $G$ is a tube that is compatible with every other tube, so it is a cone point in $\Delta_G$. The nested set complex is sometimes defined with these cone points removed since this subcomplex is a simplicial sphere. For our purposes, however, it will be convenient to consider the maximal tubes as part of every maximal tubing of $G$.


The nested set complex may be realized as a simplicial fan, which is the normal fan $\Ncal_G$ of a polytope $P_G$ known as the graph associahedron \cite[Theorem 2.6]{carr.devadoss:2006coxeter}, \cite[Theorem 3.14]{feichtner.sturmfels:2005matroid}, \cite[Theorem 7.4]{postnikov:2009permutohedra}. We recall Postnikov's construction below.

For polytopes $P,Q\subseteq\Rbb^n$, their \emph{Minkowski sum} $P+Q$ is the polytope
$$P+Q=\{\mathbf{x}+\mathbf{y}\ |\ \mathbf{x}\in P,\ \mathbf{y}\in Q\}.$$
The normal fan of $P$ is a coarsening of the normal fan of $P+Q$ \cite[Proposition~7.12]{ZieglerGu}.
Let $\mathbf{e}_1,\ldots,\mathbf{e}_n$ be the standard basis vectors in $\Rbb^n$. Given $I\subseteq[n]$, let $\Delta_I$ be the simplex with vertices $\{\mathbf{e}_i\ |\ i\in I\}$. The \emph{graph associahedron} $P_G$ is the Minkowski sum of simplices $\Delta_I$ over all tubes $I$ of $G$; that is,

$$P_G=\sum\Delta_I=\left\{\sum \mathbf{x}_I\ |\ (\mathbf{x}_I\in\Delta_I:\ I\ \mbox{is a tube})\right\}.$$

\begin{figure}
  
  \centering
  \includegraphics{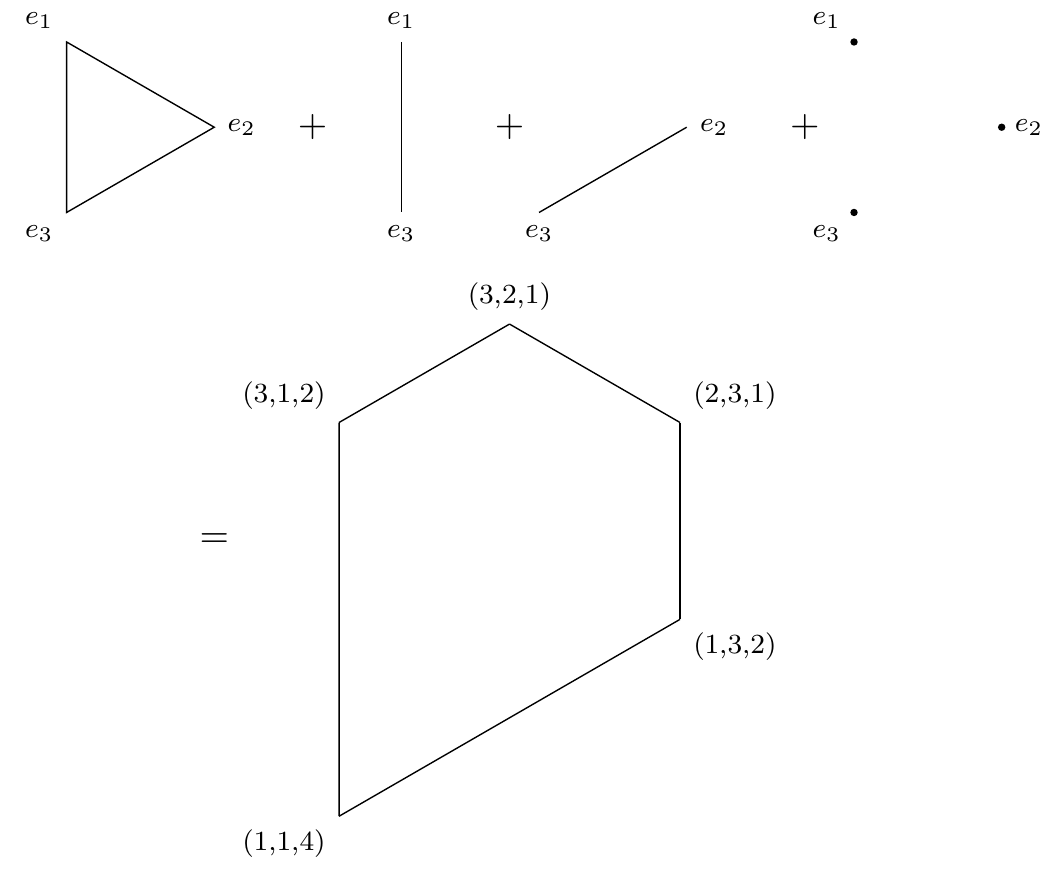}
  \caption{The graph associahedron for the graph with edge set $E=\{\{1,3\},\{2,3\}\}$}
  
\end{figure}

Proofs that the face lattice of $P_G$ coincides with the nested set complex are given in \cite{feichtner.sturmfels:2005matroid} and \cite{postnikov:2009permutohedra}.
We recall the correspondence between maximal tubings and vertices, which will be most important for our purposes.
See \cite[Proposition~7.9]{postnikov:2009permutohedra}.
Recall that the notation $i_{\down}$ refers to the principal order ideal generated by $i$ in a $G$-tree.
For a maximal tubing $\Xcal$, we interpret $i_{\down}$ as the smallest tube in $\Xcal$ that contains the element $i$.

\begin{lemma}\label{polytope_poset}
If $\Xcal$ is any maximal tubing, the point $\mathbf{v}^\Xcal=(v_1,\ldots,v_n)$ is a vertex of $P_G$ where $v_i$ is the number of tubes $I$ such that $i\in I$ and $I\subseteq i_{\down}$. Conversely, every vertex of $P_G$ comes from a maximal tubing in this way. 

\end{lemma}

Before we give the proof of Lemma~\ref{polytope_poset}, we need the following easy lemma.

\begin{lemma}\label{poset_polytope_helper}
Let $\Xcal$ be a tubing of $G$ and let $w_1\ldots w_n$ a permutation on $[n]$ such that $\{w_1,\ldots, w_j\}$ is an ideal of $\Xcal$ for each $j\in[n]$.
Suppose that $i=w_j$ for some $j\in[n]$, and write the ideal $\{w_1,\ldots, w_j\}$ as a disjoint union of tubes $X_1\cup \cdots \cup X_r$.
Then $i_\down=X_l$ for some $l\in [r]$.
\end{lemma}
\begin{proof}
Since $i_\down$ is the smallest tube in $\Xcal$ containing $i$ there is a unique $l\in [r]$ such that $i_\down \subseteq X_l$.
Assume that $i_\down$ is a proper subset of $X_l$, and choose $k\in X_l\setminus i_\down$ such that $i_\down\cup \{k\}$ is a tube.
(This is possible because $X_l$ is a tube; that is, $G|_{X_l}$ is connected.)
Since $k\in \{w_1,\ldots, w_j\}$ (and clearly $k\ne i$), there is some $p<j$ such that $w_p=k$.

Now consider the tube $k_\down\subseteq \{w_1,\ldots, w_p\}$.
Observe that $i_\down\not\subseteq k_\down$ because $i\not \in k_\down$.
Also $k_\down\not\subseteq i_\down$ because $k\notin i_\down$.
But $k_\down \cup i_\down$ is a tube (since $\{k\}\cup i_\down$ is a tube), and that is a contradiction.
The statement follows.
\end{proof}

\begin{proof}[Proof of Lemma~\ref{polytope_poset}]
  By definition, a point $\mathbf{v}\in P_G$ is a vertex if there exists a linear functional $f:\Rbb^n\ra\Rbb$ such that $\mathbf{v}$ is the unique point in $P_G$ at which $f$ achieves its maximum value. We let $P_G^f$ denote this vertex. The key observation is that if $P_G=\sum\Delta_I$ is the decomposition of the graph associahedron $P_G$ as a Minkowski sum of simplices, then $P_G^f=\sum\Delta_I^f$.
  
  If $f$ is any linear functional such that $f(\mathbf{e}_i)\neq f(\mathbf{e}_j)$ for all $i\neq j$, then $f$ is maximized at a unique vertex of the simplex $\Delta_I$ for any nonempty $I\subseteq[n]$. Namely, if $w=w_1\cdots w_n$ is the permutation of $[n]$ such that $f(\mathbf{e}_{w_1})<\cdots <f(\mathbf{e}_{w_n})$, then $\Delta_I^f=\mathbf{e}_{w_k}$ where $k$ is the maximum index such that $w_k\in I$.

Now let $\Xcal$ be a maximal tubing, and let $\mathbf{v}=\mathbf{v}^{\Xcal}$ be defined as above. Let $w=w_1\cdots w_n$ be a permutation such that $\{w_1,\ldots,w_j\}$ is an ideal of $\Xcal$ for all $j$. 
(Such a permutation exists.
For example, take any linear extension of the $G$-tree corresponding to $\Xcal$.)
Set $$f(x_1,\ldots,x_n)=x_{w_1}+2x_{w_2}+\cdots+nx_{w_n}.$$ We claim that $P_G^f=\mathbf{v}$.

Let $I$ be a tube (not necessarily in $\Xcal$), and let $i\in I$.
To verify the claim, we will show that $f|_{\Delta_I}$ is maximized at the vertex $\mathbf{e}_i$ if and only if $\Delta_I$ contributes $\mathbf{e}_i$ to $\mathbf{v}$.
That is, $f|_{\Delta_I}$ is maximized at the vertex $\mathbf{e}_i$ if and only if $I \subseteq i_\down$.
Suppose that $i=w_j$ in the permutation~$w$. 
Observe that $f|_{\Delta_I}$ is maximized at $\mathbf{e}_i$ if and only if $I\subseteq \{w_1,w_2,\ldots, w_j\}$.
Write the ideal $\{w_1,w_2,\ldots, w_j\}$ as a disjoint union $X_1\cup X_2\cup\cdots \cup X_r$ of tubes in $\Xcal$.
By Lemma~\ref{poset_polytope_helper}, $i_\down=X_l$ for some $l$. 
If $I\subseteq X_1\cup X_2\cup \cdots \cup X_r$ then $I \subseteq X_l$ because $i\in I$ and $I$ is a tube.
Clearly, if $I\subseteq i_\down =X_l$ then $I \subseteq X_1\cup\cdots \cup X_r$.
We have proved the claim that $P_G^f=\mathbf{v}$.

  Next, we prove that every vertex of $P_G$ is of the form $\mathbf{v}^{\Xcal}$ for some $\Xcal$. Let $w$ be a permutation and $f$ any linear functional such that $f(\mathbf{e}_{w_1})<\cdots<f(\mathbf{e}_{w_n})$. If there exists some maximal tubing $\Xcal$ such that $\{w_1,\ldots,w_j\}$ is an ideal of $\Xcal$ for all $j$, then we know that $P_G^f=\mathbf{v}^{\Xcal}$. Indeed, one can define a tubing $\Xcal=\{X_1,\ldots,X_n\}$ where $X_j$ is the largest tube in the subset $\{w_1,\ldots,w_j\}$ containing $w_j$.
 (That is, $X_j$ is the set of vertices of the connected component of $G|_{ \{w_1,\ldots,w_j\}}$ containing $w_j$.)
It is clear that $\Xcal$ has the desired property.
\end{proof}

If $I$ is any tube of $G$, then the subcomplex of tubings containing $I$ is isomorphic to the product of nested set complexes $\Delta_{G|_I}\times\Delta_{G/I}$. By induction, we may deduce that any face of $P_G$ is isomorphic to a product of graph associahedra.

When $G$ is a complete graph, the polytope $P_G$ is the ``standard'' permutahedron, and its normal fan $\Ncal_G$ is the set of cones defined by the braid arrangement. For a general graph $G$, the polytope $P_G$ is a Minkowski summand of the standard permutahedron, so its normal fan is coarser than that defined by the braid arrangement.

Besides the usual ordering of tubings by inclusion, there is an alternate partial order introduced by Forcey \cite{forcey:2012species} and Ronco \cite{ronco:2012tamari}. We describe the restriction of their poset to $\MTub(G)$.


Suppose that $I$ is a non-maximal tube in $\Xcal$.
Since $P_G$ is a simple polytope whose face lattice is dual to $\Delta_G$, there exists a unique tube $J$ distinct from $I$ such that $\Ycal=\Xcal\setm\{I\}\cup\{J\}$ is a maximal tubing of $G$. Define a \emph{flip} as the relation $\Xcal\ra \Ycal$ if $\topT_\Xcal(I)<\topT_{\Ycal}(J)$. We say $\Xcal\leq \Ycal$ holds if there exists a sequence of flips of maximal tubings of the form $\Xcal\ra\cdots\ra \Ycal$.

\begin{figure}
  
  \centering
  \includegraphics{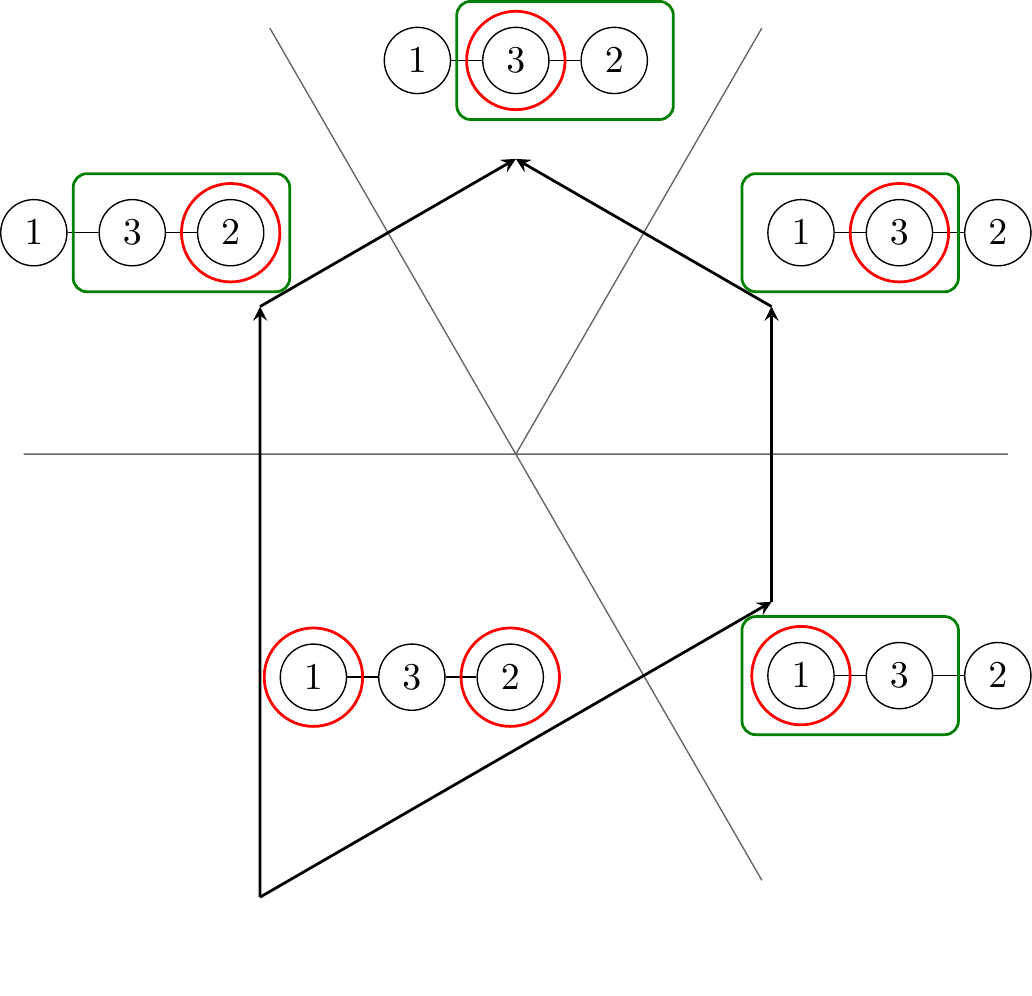}
  \caption{\label{fig:L_ex}A poset of maximal tubings}
 
\end{figure}

\begin{lemma}\label{lem_poset}
The set $\MTub(G)$ is partially ordered by the relation $\leq$.
\end{lemma}

\begin{proof}
The edges of the graph associahedron $P_G$ take the following form. Let $\Xcal$ and $\Ycal$ be maximal tubings of $G$ such that $\Ycal=\Xcal\setm\{I\}\cup\{J\}$ for some distinct tubes $I,J$. Set $i=\topT_\Xcal(I)$ and $j=\topT_{\Ycal}(J)$. Then the vertices $\mathbf{v}^{\Xcal}$ and $\mathbf{v}^{\Ycal}$ agree on every coordinate except the $i^{th}$ and $j^{th}$ coordinates. Indeed, $\mathbf{v}^{\Ycal}-\mathbf{v}^{\Xcal}=\lambda(\mathbf{e}_i-\mathbf{e}_j)$ where $\lambda$ is equal to the number of tubes of $G$ contained in $I\cup J$ that contain both $i$ and $j$. 

Let $\lambda:\Rbb^n\ra\Rbb$ such that $\lambda(x_1,\ldots,x_n)=nx_1+(n-1)x_2+\cdots+x_n$. If $\Xcal$ and $\Ycal$ are as above and $i<j$, then $\lambda(\mathbf{v}^{\Ycal}-\mathbf{v}^{\Xcal})>0$. Hence, the relation $\Xcal\ra \Ycal$ on maximal tubings is induced by the linear functional $\lambda$. Consequently, the relation is acyclic, so its transitive closure is a partial order.
\end{proof}

We let $(L_G,\leq)$ denote the poset on $\MTub(G)$ defined above. An example of the poset $L_G$ for the graph $G$ with vertex set $V=[3]$ and edge set $E=\{\{1,3\},\{2,3\}\}$ is given in Figure~\ref{fig:L_ex}. The figure demonstrates that $L_G$ is the transitive, reflexive closure of an orientation of the 1-skeleton of the graph associahedron $P_G$.

\begin{remark}
  The proof of Lemma~\ref{lem_poset} identifies the poset of maximal tubings with a poset on the 1-skeleton of the polytope $P_G$ oriented by a linear functional. This type of construction of a poset on the vertices of a polytope appears frequently in the literature, e.g. in the shellability of polytopes \cite{bruggesser.mani:1972shellable}, the complexity of the simplex method \cite{kalai:1997linear}, and the generalized Baues problem \cite{bjorner:1992essential}, among others.
  One may choose to orient the edges of $P_G$ by some other generic linear functional $\lambda^{\pr}$, giving some new partial order $L$ on the vertices of $P_G$. Letting $w=w_1\cdots w_n$ be the permutation such that $\lambda^{\pr}(\mathbf{e}_{w_1})>\cdots >\lambda^{\pr}(\mathbf{e}_{w_n})$, it is easy to see that $L\cong L_{G^{\pr}}$ where $G^{\pr}$ is the graph obtained by relabeling vertex $w_i$ in $G$ by $i$ for all $i\in[n]$. Hence, by considering the class of posets $L_G$, we are considering all posets on the vertices of a graph associahedron induced by a generic linear functional.
\end{remark}

\subsection{Properties of the poset of maximal tubings}\label{subsec_properties}


In this section, we cover some basic properties of $L_G$ that hold for any graph $G$.

If $H$ is a graph with $V(H)\subseteq\Nbb$, the \emph{standardization} $\std(H)$ is the same graph with vertex set $V(\std(H))=[n],\ n=|V(H)|$ such that the vertices of $\std(H)$ have the same relative order as in $H$. That is, there is a graph isomorphism $\phi:H\ra\std(H)$ such that if $i,j\in V(H),\ i<j$ then $\phi(i)<\phi(j)$.

\begin{lemma}\label{lem_decomposition}
  Let $G$ be a graph, $I\subseteq V(G)=[n]$ such that $G$ does not have any edge $\{i,j\}$ with $i\in I$ and $j\in[n]\setm I$. Then
  $$L_G\cong L_{\std(G|_I)}\times L_{\std(G|_{[n]\setm I})}.$$
\end{lemma}

\begin{proof}
  Under the assumptions about $G$ and $I$, there do not exist any tubes $X$ such that $X\cap I\neq\emptyset$ and $X\cap([n]\setm I)\neq\emptyset$. Furthermore, any tube of $G|_I$ is compatible as a tube of $G$ with any tube of $G|_{[n]\setm I}$. Hence, the set $\MTub(G)$ naturally decomposes as a Cartesian product
  $$\MTub(G)\stackrel{\sim}{\longlra}\MTub(\std(G|_I))\times\MTub(\std(G|_{[n]\setm I})).$$
  We claim that this bijection induces the desired isomorphism of posets
  $$L_G\cong L_{\std(G|_I)}\times L_{\std(G|_{[n]\setm I})}.$$

  If $\Xcal,\Ycal\in\MTub(G)$ such that $\Ycal=\Xcal\setm\{J\}\cup\{J^{\pr}\}$ for some tubes $J,J^{\pr}$ then $J$ and $J^{\pr}$ must be incompatible. Consequently, either both tubes $J,J^{\pr}$ are contained in $I$, or both tubes are contained in $[n]\setm I$. Without loss of generality, assume that $\topT_{\Xcal}(J)<\topT_{\Ycal}(J^{\pr})$ and that $J$ and $J^{\pr}$ are both subsets of $I$, which implies $\Xcal\ra\Ycal$ holds in $L_G$.

  Let $\phi:H\ra\std(H)$ be the natural graph isomorphism between $H$ and its standardization. Then the inequality $\topT_{\std(\Xcal|_I)}(\phi(J))<\topT_{\std(\Xcal|_I)}(\phi(J^{\pr}))$ still holds, so we have the relation $\std(\Xcal|_I)\ra\std(\Ycal|_I)$ in $L_{\std(G|_I)}$.

  Conversely, if $\Xcal$ and $\Ycal$ are maximal tubings of $\std(G|_I)$ and $\Zcal$ is any maximal tubing of $G|_{[n]\setm I}$, then the relation $\Xcal\ra\Ycal$ in $L_{\std(G|_I)}$ implies a relation $(\phi^{-1}(\Xcal)\cup\Zcal)\ra(\phi^{-1}(\Ycal)\cup\Zcal)$ in $L_G$.
\end{proof}


If $(L,\leq)$ is any poset, its \emph{dual} $(L^*,\leq^*)$ is the poset with the same underlying set such that $a\leq b$ if and only if $b\leq^* a$. 
If $G$ is any graph with $V(G)=[n]$, we let $G^*$ be the graph obtained by swapping vertices $i$ and $n+1-i$ for all $i$. 
This induces a natural bijection between maximal tubings of $G$ and maximal tubings of $G^*$.

\begin{lemma}\label{graph duality}
  The natural bijection $\MTub(G)\ra\MTub(G^*)$ induces an isomorphism of posets $L_G^*\cong L_{G^*}$.
\end{lemma}

\begin{proof}
  Let $\Xcal,\Ycal\in\MTub(G)$ are distinct tubings such that $\Ycal=\Xcal\setm\{I\}\cup\{J\}$. Let $\Xcal^*,\Ycal^*\in\MTub(G^*)$ be the corresponding maximal tubings of $G^*$. Then
  \begin{align*}
    \Xcal\ra\Ycal &\LRa \top_{\Xcal}(I)<\top_{\Ycal}(J)\\
    &\LRa \top_{\Ycal^*}(J^*)<\top_{\Xcal^*}(I^*)\\
    &\LRa \Ycal^*\ra\Xcal^*.
  \end{align*}
 Passing to the transitive closure of $\ra$, we deduce that $L_G$ and $L_{G^*}$ are dual posets.
\end{proof}



\subsection{The non-revisiting chain property}\label{subsec:NRC}

In this section, we prove that graph associahedra have the non-revisiting chain property, defined below. This is equivalent to the statement that for any tubing $\Xcal$, the set of maximal tubings containing $\Xcal$ is an interval of~$L_G$.

Given a polytope $P$, we will say a linear functional $\lambda:\Rbb^n\ra\Rbb$ is \emph{generic} if it is not constant on any edge of $P$. When $\lambda$ is generic, we let $L(P,\lambda)$ be the poset on the vertices of $P$ where $v\leq w$ if there exists a sequence of vertices $v=v_0,v_1,\ldots,v_l=w$ such that $\lambda(v_0)<\lambda(v_1)<\cdots<\lambda(v_l)$ and $[v_{i-1},v_i]$ is an edge for all $i\in\{1,\ldots,l\}$.

The following properties of $L(P,\lambda)$ are immediate.

\begin{proposition}\label{prop:omega_properties}
  Let $P$ be a polytope with a generic linear functional $\lambda$.
  \begin{enumerate}
  \item The dual poset $L(P,\lambda)^*$ is isomorphic to $L(P,-\lambda)$.
  \item If $F$ is a face of $P$, then the inclusion $L(F,\lambda)\hookra L(P,\lambda)$ is order-preserving.
  \item $L(P,\lambda)$ has a unique minimum $v_{\hat{0}}$ and a unique maximum $v_{\hat{1}}$.
  \end{enumerate}
\end{proposition}

The pair $(P,\lambda)$ is said to have the \emph{non-revisiting chain (NRC) property} if whenever $\mathbf{x}<\mathbf{y}<\mathbf{z}$ in $L(P,\lambda)$ such that $\mathbf{x}$ and $\mathbf{z}$ lie in a common face $F$, then $\mathbf{y}$ is also in $F$.
The name comes from the fact that if $P$ has the NRC property, then any sequence of vertices following edges monotonically in the direction of $\lambda$ does not return to a face after leaving it. 
By definition, the NRC property means that faces are \emph{order-convex} subsets of $L(P,\lambda)$.
(Recall that a subset $S$ of a poset is \emph{order-convex} provided that whenever elements $x,z\in S$ satisfy $x<z$ then the entire interval $[x,z]$ belongs to $S$.)
In light of Proposition~\ref{prop:omega_properties}, this is equivalent to the condition that for any face $F$, the set of vertices of $F$ form an interval of $L(P,\lambda)$ isomorphic to $L(F,\lambda)$.

\begin{remark}
  There is also an unoriented version of the NRC property due to Klee and Wolfe called the \emph{non-revisiting path property}, which is the condition that for any two vertices $\mathbf{v}, \mathbf{w}$ of $P$, there exists a path from $\mathbf{v}$ to $\mathbf{w}$ that does not revisit any facet of $P$. 
It was known that the Hirsch conjecture on the diameter of 1-skeleta of polytopes is equivalent to the conjecture that every polytope has the non-revisiting path property.
These conjectures were formulated to determine the computational complexity of the simplex method from linear programming in the \emph{worst-case} scenario.
The Hirsch conjecture was disproved by Santos \cite{santos:2012counterexample}, but many interesting questions remain.
In particular, the polynomial Hirsch conjecture remains open.
\end{remark}

In contrast to the non-revisiting path property, many low-dimensional polytopes lack the non-revisiting chain property.
For example, if $P$ is a simplex of dimension at least $2$, then $[\mathbf{v}_{\hat{0}},\mathbf{v}_{\hat{1}}]$ is an edge of $P$ that is not an interval of $L(P,\lambda)$.
However, the property does behave nicely under Minkowski sum.

\begin{proposition}\label{prop:MS_NRF}
  If $(P,\lambda)$ and $(Q,\lambda)$ have the non-revisiting chain property, then so does $(P+Q,\lambda)$.
\end{proposition}

The proof of Proposition~\ref{prop:MS_NRF} relies on Lemma~\ref{lem:sum_order_embed}. For polytopes $P$ and $Q$, the normal fan of $P+Q$ is the common refinement of $\Ncal(P)$ and $\Ncal(Q)$; that is,
$$\Ncal(P+Q)=\{C\cap C^{\pr}\ |\ C\in\Ncal(P),\ C^{\pr}\in\Ncal(Q)\}.$$
Let $V(P)$ be the set of vertices of $P$, and let $C_v$ be the normal cone to the vertex $v$ in $P$. From the description of the normal fan of $P+Q$, there is a canonical injection $\iota:V(P+Q)\hookra V(P)\times V(Q)$ that assigns a vertex $\mathbf{v}\in P+Q$ to $(\mathbf{u},\mathbf{w})$ if the normal cones satisfy $C_{\mathbf{v}}=C_\mathbf{u}\cap C_\mathbf{w}$.

\begin{lemma}\label{lem:sum_order_embed}
  The map $\iota:V(P+Q)\hookra V(P)\times V(Q)$ is an order-preserving function from $L(P+Q,\lambda)$ to $L(P,\lambda)\times L(Q,\lambda)$.
\end{lemma}

\begin{proof}
  Let $E=[\mathbf{v},\mathbf{w}]$ be an edge of $P+Q$, and suppose $\lambda(\mathbf{v})<\lambda(\mathbf{w})$. It suffices to show that $\iota(\mathbf{v})<\iota(\mathbf{w})$.

  Let $\iota(\mathbf{v})=(\mathbf{v}^{\pr},\mathbf{v}^{\pr\pr})$ and $\iota(\mathbf{w})=(\mathbf{w}^{\pr},\mathbf{w}^{\pr\pr})$. 
 Then the normal cone $C_E$ is the intersection of $C_\mathbf{v}$ and $C_\mathbf{w}$, which themselves are the intersections of $C_{\mathbf{v}^{\pr}},\ C_{\mathbf{v}^{\pr\pr}}$ and $C_{\mathbf{w}^{\pr}},\ C_{\mathbf{w}^{\pr\pr}}$. 
 Since
  $$C_E=(C_{\mathbf{v}^{\pr}}\cap C_{\mathbf{w}^{\pr}})\cap(C_{\mathbf{v}^{\pr\pr}}\cap C_{\mathbf{w}^{\pr\pr}})$$
  is a cone of codimension 1, we may deduce that $C_{\mathbf{v}^{\pr}}\cap C_{\mathbf{w}^{\pr}}$ and $C_{\mathbf{v}^{\pr\pr}}\cap C_{\mathbf{w}^{\pr\pr}}$ are both of codimension $\leq 1$. 
  Hence, the segments $E^{\pr}=[\mathbf{v}^{\pr},\mathbf{w}^{\pr}]$ and $E^{\pr\pr}=[\mathbf{v}^{\pr\pr},\mathbf{w}^{\pr\pr}]$ are either vertices or edges of $P$ and $Q$, respectively. 
  Moreover, if both $E^{\pr}$ and $E^{\pr\pr}$ are edges, then they must be parallel and $E=E^{\pr}+E^{\pr\pr}$. 
  In the event one of them is a vertex, say $E^{\pr\pr}$ (so that $\mathbf{v}''=\mathbf{w}''$), then $E^{\pr}$ must be an edge, and
  $$\lambda(\mathbf{v}^{\pr})=\lambda(\mathbf{v})-\lambda(\mathbf{v}^{\pr\pr})<\lambda(\mathbf{w})-\lambda(\mathbf{v}^{\pr\pr})=\lambda(\mathbf{w})-\lambda(\mathbf{w}^{\pr\pr})=\lambda(\mathbf{w}^{\pr}).$$
  If both $E^\pr$ and $E^{\pr\pr}$ are edges, then since $\lambda$ achieves its minimum value on $E=E^{\pr}+E^{\pr\pr}$ at $\mathbf{v}=\mathbf{v}^{\pr}+\mathbf{v}^{\pr\pr}$, we have $\lambda(\mathbf{v}^{\pr})<\lambda(\mathbf{w}^{\pr})$ and $\lambda(\mathbf{v}^{\pr\pr})<\lambda(\mathbf{w}^{\pr\pr})$. 
  In both cases, $\iota(\mathbf{v})<\iota(\mathbf{w})$ holds.
\end{proof}

\begin{proof}[Proof of Proposition~\ref{prop:MS_NRF}]
Every face of $P+Q$ is of the form $F+F^{\pr}$ where $F$ is a face of $P$ and $F^{\pr}$ is a face of $Q$. 
Suppose $\mathbf{u},\mathbf{v},\mathbf{w}$ are vertices of $P+Q$ such that $\mathbf{u}<\mathbf{v}<\mathbf{w}$ in $L(P+Q,\lambda)$ and $\mathbf{u},\mathbf{w}\in F+F^{\pr}$. Set $\iota(\mathbf{u})=(\mathbf{u}_P,\mathbf{u}_Q)$, and analogously for $\iota(\mathbf{v})$ and $\iota(\mathbf{w})$. 
Then $\mathbf{u}_P\leq \mathbf{v}_P\leq \mathbf{w}_P$ in $L(P,\lambda)$ and $\mathbf{u}_Q\leq \mathbf{v}_Q\leq \mathbf{w}_Q$ in $L(Q,\lambda)$. 
Since $P$ and $Q$ have the non-revisiting chain property, $\mathbf{v}_P$ is in $F$ and $\mathbf{v}_Q$ is in $F^{\pr}$.
 Hence, $\mathbf{v}=\mathbf{v}_P+\mathbf{v}_Q$ is in $F+F^{\pr}$, as desired.
\end{proof}

\begin{corollary}[Proposition 7.2 \cite{hersh:2018nonrevisiting}]
  Every zonotope has the non-revisiting chain property with respect to any generic linear functional.
\end{corollary}

We now return to graph associahedra. 
Let $G$ be a graph on $[n]$, and let $\lambda$ be the linear functional in the proof of Lemma~\ref{lem_poset}, where $\lambda(\mathbf{x})=nx_1+(n-1)x_2+\cdots+x_n$, so that $L_G\cong L(P_G,\lambda)$. Using the decomposition $P_G=\sum\Delta_I$, Lemma~\ref{lem:sum_order_embed} implies that $\pi_J:L_G\ra L(\Delta_J,\lambda)$ obtained as the composition
$$L_G\hookra\bigotimes_I L(\Delta_I,\lambda)\thra L(\Delta_J,\lambda)$$
is order-preserving. We note that the poset $L(\Delta_J,\lambda)$ is a chain where $\mathbf{e}_i>\mathbf{e}_j$ whenever $i,j\in J$ with $i<j$. 
\begin{lemma}\label{NRC helper}
Suppose that $\Xcal$ is a maximal tubing of $G$ and $J$ is a tube not necessarily in~$\Xcal$.
Then there exists a unique $k\in J$ such that $J\subseteq k_{\down}$, and for this $k$ we have $\pi_J(\Xcal)=\mathbf{e}_k.$
\end{lemma}
\begin{proof}
Recall that $k_\down$ is the smallest tube in $\Xcal$ that contains $k$.
Hence, there is at most one such element $k\in J$ satisfying $J\subseteq k_\down$.
(Indeed, if $j\in J$ and $J\subseteq j_\down$ then we have $j\in J\subseteq k_\down$.
Thus $j_\down\subseteq k_\down$.
By symmetry, $k_\down\subseteq j_\down$.
Therefore $j=k$.)

Consider the vertex $\mathbf{v}^\Xcal$ in $P_G$.
Lemma~\ref{polytope_poset} implies that $\Delta_J$ contributes $\mathbf{e}_k$ to $\mathbf{v}^\Xcal$ if and only if $k\in J$ and $J\subseteq k_\down$.
Therefore $\pi_J(\Xcal) = \mathbf{e}_k$, as desired.

\end{proof} 

\begin{theorem}\label{thm:NRC}
  The pair $(P_G,\lambda)$ has the non-revisiting chain property.
\end{theorem}

\begin{proof}
  Every face of $P_G$ is the intersection of some facets, and the intersection of a family of order-convex sets is again order-convex. Hence, it suffices to prove that if $F$ is any facet of $P_G$ then $V(F)$ is an order-convex subset of $L(P_G,\lambda)$. We argue by way of contradiction that this set is order-convex by selecting an appropriate projection $\pi_J$.

  Under the dictionary between tubings of $G$ and faces of $P_G$, if $F$ is a facet, then there exists a tube $I$ such that
  $$V(F)=\{\mathbf{v}^{\Xcal}\ |\ \Xcal\in\MTub(G),\ I\in\Xcal\}.$$
  Suppose that there are maximal tubings $\Xcal<\Ycal<\Zcal$ such that $I\in\Xcal$ and $I\in\Zcal$ but $I$ is not in $\Ycal$. Given that such a triple exists, we are free to assume that $\Xcal\ra\Ycal$ is a flip. Then the flip exchanges $I$ for some tube $I^{\pr}$. Let $a=\topT_{\Xcal}(I)$ and $b=\topT_{\Ycal}(I^{\pr})$. The union $I\cup I^{\pr}$ is a tube in both $\Xcal$ and $\Ycal$ such that $b=\topT_{\Xcal}(I\cup I^{\pr})$ and $a=\topT_{\Ycal}(I\cup I^{\pr})$. 
Hence, $I$ is maximal a tube in the ideal $(I\cup I^{\pr})\setminus \{b\}$ in $\Xcal$.
That means $G|_I$ is one of the connected components of $G|_{I\cup I^{\pr}\setm\{b\}}$. Since $I\cup I^{\pr}$ is a tube, this implies $I\cup\{b\}$ is a tube as well.

 Set $J=I\cup\{b\}$. 
We claim that if $\Wcal$ is any maximal tubing containing $I$, then the projection $\pi_J(\Wcal)=\mathbf{e}_b$. 
If $\pi_J(\Wcal)=\mathbf{e}_k\ne \mathbf{e}_b$ then Lemma~\ref{NRC helper} says that $k\in J$ and $J \subseteq k_\down$.
Since $k\ne b$, it follows that $k\in I$.
Since $k_\down$ is the smallest tube in $\Wcal$ that contains $k$, we have $k_\down \subseteq I$.
But then $I\subsetneq J \subseteq k_\down \subseteq I$, and that is a contradiction.
Therefore, $\pi_J(\Wcal) = \mathbf{e}_b$.

So $\pi_J(\Xcal)=\mathbf{e}_b=\pi_J(\Zcal)$, but $\pi_J(\Ycal)=\mathbf{e}_a$, contradicting the assumption that $\Ycal<\Zcal$.
\end{proof}

\begin{corollary}\label{faces_are_intervals}
For any tubing $\Ycal$ of $G$, the set of maximal tubings which contain $\Ycal$ is an interval in $L_G$.
\end{corollary}

\begin{remark}
  Another property that a polytope graph may have is the \emph{non-leaving face property}, which is satisfied if for any two vertices $u,v$ that lie in a common face $F$ of $P$, every geodesic between $u$ and $v$ is completely contained in $F$. This property holds for all zonotopes, but is quite special for general polytopes. Although ordinary associahedra are known to have the non-leaving face property, not all graph associahedra do. We note that the example geodesic in \cite[Figure 6]{manneville.pilaud:2015graph} that leaves a particular facet cannot be made into a monotone path, so it does not contradict our Theorem~\ref{thm:NRC}.
\end{remark}

Recall that the M\"obius function $\mu=\mu_L:\Int(L)\ra\Zbb$ is the unique function on the intervals of a finite poset $L$ such that for $x\leq y$:
$$\sum_{x\leq z\leq y}\mu(x,z)=\begin{cases}1\ \mbox{if }x=y\\0\ \mbox{if }x\neq y\end{cases}.$$

When $L(P,\lambda)$ is a lattice with the non-revisiting chain property, the M\"obius function was determined in \cite{hersh:2018nonrevisiting}. One way to prove this is to show that $L(P,\lambda)$ is a crosscut-simplicial lattice; cf. \cite{mcconville:2017crosscut}. In the case of the poset of maximal tubings, we may express the M\"obius function as follows. For a tubing $\Xcal$, let $|\Xcal|$ be the number of tubes it contains.

\begin{corollary}\label{cor:mobius}
  Let $G$ be a graph with vertex set $[n]$ such that $L_G$ is a lattice. Let $\Xcal$ be a tubing that contains every maximal tube. 
 The set of maximal tubings containing $\Xcal$ is an interval $[\Ycal,\Zcal]$ of $L_G$ such that $\mu(\Ycal,\Zcal)=(-1)^{n-|\Xcal|}$. If $[\Ycal,\Zcal]$ is not an interval of this form, then $\mu(\Ycal,\Zcal)=0$.
\end{corollary}

Based on some small examples, we conjecture that Corollary~\ref{cor:mobius} is true even without the assumption that $L_G$ is a lattice.

\subsection{Covering relations and $G$-forests}\label{subsec_Gtrees}

As above, let $G$ be a graph with vertex set $[n]$.
In the following sections, it will be useful to realize $L_G$ as a partial order on the set of $G$-forests.
The advantage to working with $G$-forests, rather than maximal tubings, is that cover relations in $L_G$ are encoded by certain adjacent (covering) pairs in the forest poset.

As in Theorem~\ref{G-trees}, let $T$ be a $G$-forest and let $\Xcal$ be the maximal tubing $\chi(T)$.
Recall that we write $i<_T k$ if $k$ is in the unique path from $i$ to the root.
A covering relation \emph{in} $T$ is a pair $i$ and $k$ such that $i<_Tk$ and also, $i$ and $k$ are adjacent in $T$.
We say that $k$ \emph{covers} $i$ (or $i$ \emph{is covered by} $k$) and write $i\covered_T k$ (or $k\covers_T i$).
We say that $k$ has a \emph{lower (resp. upper) cover} if there exists an element $i\in T$ such that $i \covered_Tk$ (resp. $i\covers_T k)$.
The following easy lemma will be useful.

\begin{lemma}\label{parent}
Let $T$ be a $G$-forest or $G$-forest.
Each element in $T$ has at most one upper cover.
In particular, if $i<_T j$ and $i<_T k$, then $j$ and $k$ are comparable.
\end{lemma}
\begin{proof}
Suppose that $i$ is less than $j$ and $k$ in $T$.
Then the tubes $j_\down$ and $k_\down$ have nonempty intersection.
Thus, either $j_\down\subseteq k_\down$ or $k_\down \subseteq j_\down$.
\end{proof}

We say that the pair $(i,k)$ is a \emph{descent of $T$} if $k\covered_T~i$ and $i<k$ as integers.
Dually, the pair is an \emph{ascent of $T$} if $i>k$ as integers.
The next proposition follows from Theorem~\ref{G-trees}.
\begin{proposition}\label{cor: covering relations}
Suppose that $T$ is a $G$-forest and $\Xcal$ is the corresponding maximal tubing~$\chi(T)$.
\begin{itemize}
\item Each descent $(i,k)$ in $T$ corresponds bijectively to a covering relations $\Xcal \covers \Xcal'$ in~$L_G$.
\item Each ascent $(i,k)$ in $T$ corresponds bijectively to a covering relation $\Xcal'' \covers \Xcal$ in $L_G$.
\end{itemize}
\end{proposition}


\begin{proposition}\label{cover relations}
Let $T$ be a $G$-forest with descent $(i,k)$, and let $\Xcal=\chi(T)$ be its corresponding maximal tubing.
Write the ideal $\{x: x<_{T} k\}$ as the disjoint union of tubes $Y_1\cup \cdots \cup Y_t$.
Then swapping $i$ and $k$, we obtain a $G$-forest covered by $T$ in $L_G$, whose corresponding maximal tubing is  $$\Xcal\setminus \{k_\down\} \cup \left\{ i_\down \setminus \left(\{k\} \cup \bigcup Y_{a_j}\right) \right\},$$
where the union $\bigcup Y_{a_j}$ is over all $Y_{a_j}\in \{Y_1,\ldots, Y_t\}$ such that $Y_{a_j}\cup \{i\}$ not a tube.
(Throughout $x_\down$ is interpreted as the principal order ideal in $T$.)
\end{proposition}
\begin{proof}
Write $S$ for $ i_\down \setminus \left(\{k\} \cup \bigcup Y_{a_j}\right)$ and $\Ycal$ for $\Xcal\setminus \{k_\down\} \cup \{S\}$.
First we show that $\Ycal$ is a maximal tubing.
Observe that $S$ is a tube.
We check that each tube $I$ in $\Xcal\setminus \{k_\down\}$ is compatible with $S$.
Since both $I$ and $i_\down$ are tubes in $\Xcal$, either $I\subset i_\down$, $I\supset i_\down$, or $I\cup i_\down$ is not a tube.
If $I\supset i_\down$ or $I\cup i_\down$ is not a tube, then the fact that $S\subset i_\down$ implies that $I$ and $S$ are compatible.
So, we assume that $I$ is a subset of $i_\down$.
Write $X_1\cup X_2\cup \cdots \cup X_r$ for the ideal $\{x: x<_{T} i\}$.
Since $i\covers_{T} k$ we have $k_\down= X_l$ for some $l$.
Thus $I\subseteq X_j$ for some $j\ne l$ or  $I\subseteq Y_s$ for some $s\in [t]$.
If $I\subseteq X_j$ then it follows immediately that $I\subseteq S$.
Similarly, if $I\subseteq Y_s$ and $Y_s\cup \{i\}$ is a tube, then $I\subseteq S$.

Assume that $I\subseteq Y_s$ and $Y_{s}\cup \{i\}$ is not a tube.
Then $I\not\subseteq S$ and $S\not\subseteq I$. 
We claim that that $Y_s\cup S$ is not a tube. 

Observe that $X_j$ and $Y_s$ are compatible in $\Xcal$, and neither $X_j\not\subset Y_s$ nor $Y_s\not\subseteq X_j$, for each $j\in [r]$ with $j\ne l$.
Thus, $X_j\cup Y_s$ is not a tube.
The same argument shows each tube $Y_s\cup Y_j$ is not a tube, for each $j\in[t]$ with $j\ne s$.
Thus $Y_s\cup S$ is not a tube, and hence $I\cup S$ is not a tube.
We conclude that $I$ and $S$ are compatible.

We conclude that $\Ycal$ is a maximal tubing of $G$.
Since $\Ycal$ differs from $\Xcal$ by a flip, it follows that $\Xcal$ covers $\Ycal$ in~$L_G$.
\end{proof}

\section{Lattices}\label{sec:lattice}
\subsection{Lattices and lattice congruences}
Recall that a poset $L$ is a lattice if each pair $x$ and $y$
has a greatest common lower bound or \emph{meet} $x\meet y$, and 
has a smallest common upper bound or \emph{join} $x\join y$.
Throughout we assume that $L$ is finite.

A set map $\phi: L\to L'$ is a \emph{lattice map} if it satisfies  $\phi(x\meet y) = \phi(x)\meet \phi(y)$ and  $\phi(x\join y) = \phi(x)\join \phi(y)$.
We say that $\phi$ preserves both the meet and join operations.
When $\phi$ is surjective, we say that it is a \emph{lattice quotient map} and $L'$ is a lattice quotient of $L$.
We say that $\phi$ is \emph{meet (join) semilattice map} if it preserves the meet (join) operation, and the image $\phi(L)$ is called a \emph{meet (join) semilattice quotient} of $L$.

To determine whether a given set map $\phi: L\to L'$ preserves either the meet or join operations, we consider the equivalence relation on $L$ induced by the fibers of $\phi$.
That is, set $x\equiv y \mod \Theta_\phi$ if $\phi(x)=\phi(y)$.

\begin{definition}\label{def: cong}
Let $L$ be a finite lattice, and let $\Theta$ be an equivalence relation on $L$.
We say that $\Theta$ is a \emph{lattice congruence} if it satisfies both of the following conditions for each $x,y,$ and $z$ in $L$.
\begin{equation}\label{meet-preserving}
\text{ if $x\equiv y \mod \Theta$ then $x\meet z\equiv y\meet z\mod \Theta$}\tag{$\Mcal$}
\end{equation}

\begin{equation}\label{join-preserving}
\text{ if $x\equiv y \mod \Theta$ then $x\join z\equiv y\join z\mod \Theta$}\tag{$\Jcal$}
\end{equation}
We say that $\Theta$ is a \emph{meet (join) semilattice congruence} if $\Theta$ satisfies~\ref{meet-preserving} (\ref{join-preserving}).
\end{definition}

Observe that $\phi: L\to L'$ preserves the meet (join) if and only if the equivalence relation $\Theta_\phi$ induced by its fibers is a meet (join) semilattice congruence.
The next proposition implies that each meet semilattice  congruence on $L$ gives rise to a meet semilattice quotient. 
\begin{proposition}\label{meet cong}
Let $\Theta$ be an equivalence relation on $L$.
Then $\Theta$ is a meet semilattice congruence if and only if $L$ satisfies each of the following conditions:
\begin{enumerate}
\item Each $\Theta$-class has a unique minimal element;
\item the map $\pidown^\Theta:L\to L$ which sends $x$ to the unique minimal element in its $\Theta$-class is order preserving.
\end{enumerate}
In particular, the subposet of $L$ induced by $\pidown^\Theta(L)$ is a meet semilattice quotient of $L$.
\end{proposition}
\begin{proof}
The proof of the first statement can be found in \cite[Proposition~9-5.2]{reading:2016lattice}.
We assume that $\Theta$ is a meet semilattice congruence or, equivalently, that the two conditions above hold.
We check that the subposet of $L$ induced by the image $\pidown^{\Theta}(L)$ is a lattice and that $\pidown^{\Theta}$ is a meet semilattice map.

Suppose that $x$ and $y$ belong to $\pidown^\Theta(L)$.
We write $x\meet_\Theta y$ to distinguish the meet operation in $\pidown^\Theta(L)$ from the meet operation in $L$.
(In general, these are different operations; that is, $x\meet_\Theta y \ne x\meet y$.)
It is enough to show that the meet $x\meet_{\Theta} y$ is equal to $\pidown^{\Theta}(x\meet y)$.

Because $\pidown^\Theta$ is order preserving, we have $\pidown^{\Theta}(x\meet y)\le x, y$.
If $z\in \pidown^\Theta(L)$ and $z$ is a common lower bound for $x$ and $y$ then $z\le x\meet y$.
Applying the fact that $\pidown^\Theta$ is order preserving again, we have $z=\pidown^\Theta(z) \le \pidown^\Theta(x\meet y)$.
\end{proof}
%
%

The set $\Con(L)$ of lattice congruences of $L$ forms a distributive lattice under the refinement order.
That is, $\Theta\leq\Theta^{\pr}$ holds if $x\equiv y\mod\Theta^{\pr}$ implies $x\equiv y\mod\Theta$ for $x,y\in L$. 
Hence, when $\Con(L)$ is finite, it is the lattice of order ideals of its subposet of join-irreducible elements. 
If $L$ is a lattice with a cover relation $x\lessdot y$, the \emph{contraction} $\con(x,y)$ is the most refined lattice congruence identifying $x$ and $y$.
It is known that $\con(x,y)$ is join-irreducible, and if $L$ is finite, then every join-irreducible lattice congruence is of this form \cite[Proposition 9-5.14]{reading:2016lattice}.

\subsection{Lattice congruences of the weak order}\label{subsec_weak_cong}
Recall $x\le y$ in the weak order on $\Sfrak_n$ if $\inv(x) \subseteq \inv(y)$, where $\inv(x)$ is the set of inversions of $x$.
(A pair $(i,k)$ is an \emph{inversion} of $x$ if $i<k$, and $k$ precedes $i$ in $x=x_1\ldots x_n$.
That is, $i=x_s$ and $k=x_r$, where $r<s$.)
It is well-known that the weak order on $\Sfrak_n$ is a lattice.

A \emph{descent} of $x$ is an inversion $(i,k)$ such that $i$ and $k$ are consecutive in $x_1\ldots x_n$.
That is, $i=x_s$ and $k=x_{s-1}$, where $s\in \{2,\ldots, n\}$.
The \emph{descent set} $\des(x)$ of $x$ is the set of all descents of $x$.
An \emph{ascent} is a noninversion $(i,k)$ in which $i=x_{s-1}$ and $k=x_{s}$.

If $y_s=i$ and $y_{s-1}=k$ is a descent of $y$, then swapping the positions of $i$ and $k$, we obtain a permutation $x$ (with $x_i=y_i$ for each $i\in [n]\setminus \{s-1, s\}$ and $x_{s-1}=i$ and $x_s=k$) that is covered by $y$ in the weak order.
Each lower cover relation $y\covers x$ corresponds bijectively to a descent of $y$.
Dually, each upper cover relation $y\covered y'$ corresponds bijectively to an ascent of $y$.
The following lemma is immediate.

\begin{lemma}\label{descents and inversions}
Suppose that $x'> x$ in the weak order on $\Sfrak_n$.
Then there exists a descent $(i,k)$ of $x'$ which is not an inversion of $x$.
Swapping $i$ and $k$ in $x'$ we obtain a permutation $x''$ which satisfies $x'\covers x''\ge x.$
\end{lemma}

Recall that each pair $x\covered y$ maps to a join-irreducible congruence $\con(x,y)$ in $\Con(\Sfrak_n)$.
For $n>2$, this map is not injective.
We can obtain a bijection by restricting to pairs $x\covered y$ where $y$ is join-irreducible.
Below, we make this bijection explicit with the combinatorics of arc diagrams.

An \emph{arc} is a triple $\alpha=(i,k, \epsilon)$ where $1\leq i<k\leq n$ and $\epsilon=(\epsilon_1,\ldots,\epsilon_{k-i-1})$ such that $\epsilon_h\in\{+,-\}$ for $h\in[k-i-1]$. 
Listing the numbers $1,\ldots,n$ vertically, an arc is typically drawn as a path from $i$ to $k$ that goes to the left of $j$ if $\epsilon_{j-i}=+$ and to the right of $j$ if $\epsilon_{j-i}=-$.

\begin{figure}

  \centering
  \includegraphics{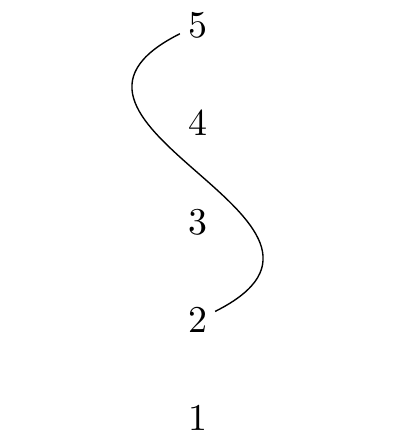}
  \caption{\label{fig_arc}The arc $(2,5,(-,+))$}
  
\end{figure}

If $x\lessdot y$ is a cover relation of permutations that swaps $i$ and $k$, then we define $\alpha(x,y)=(i,k,\epsilon)$ to be the arc such that for $i<j<k$:
$$\epsilon_{j-i}=\begin{cases}+\ \mbox{if }u^{-1}(j) > u^{-1}(i)\\-\ \mbox{if }u^{-1}(j) < u^{-1}(k)\end{cases}.$$
For example, $\alpha(32514,35214)=(2,5,(-,+))$ is the arc in Figure~\ref{fig_arc}.

Given an arc $(x,y,\epsilon)$, write $\{l_1<l_2<\ldots < l_p\}$ for the set $\{x':\epsilon_{x'-x} = -\}$ and $\{r_1<r_2<\ldots<r_q\}$ for the set $\{y': \epsilon_{y'-x} = +\}$.
Informally, $l_1<\ldots<l_p$ are the nodes on the left side of the arc $(x,y,\epsilon)$, and $r_1<\ldots<r_q$ are the nodes on the right side of the arc.
The next results follows from \cite[Proposition~2.3]{reading:2015noncrossing}.
 
\begin{proposition}\label{arc to perm}
Let $\alpha=(x,y,\epsilon)$ be an arc, and let $l_1<\ldots<l_p$ and $r_1<\ldots<r_q$ be defined as above.
Then, among all permutations $w\in \Sfrak_n$ such that $\alpha(u,w) =\alpha$ for some $u$ covered by $w$, the unique minimal element is
\[j_\alpha= 12\ldots (x-1) l_1\ldots l_p \,y\, x\, r_1\ldots r_q (y+1) (y+2) \ldots n\]
In particular, $j_\alpha$ is join-irreducible.
\end{proposition}

The map $\alpha$ induces a bijection between join-irreducible lattice congruences and join-irreducible permutations.

\begin{theorem}\label{thm_weak_arcs}
Given two cover relations $x\lessdot y$ and $x^{\pr}\lessdot y^{\pr}$, we have $\con(x,y)=\con(x^{\pr},y^{\pr})$ if and only if $\alpha(x,y)=\alpha(x^{\pr},y^{\pr})$.
\end{theorem}

In light of Theorem~\ref{thm_weak_arcs}, we will identify a join-irreducible lattice congruence $\Theta^{\alpha}$ of the weak order by its associated arc $\alpha$. For arcs $\alpha,\beta$, we say that $\alpha$ \emph{forces} $\beta$ if $\Theta^{\beta}\leq\Theta^{\alpha}$. An arc $\alpha=(i,k,\epsilon)$ is a \emph{subarc} of $\beta=(i^{\pr},k^{\pr},\epsilon^{\pr})$ if $i^{\pr}\leq i<k\leq k^{\pr}$ and for all $j\in[k-i-1]$, $\epsilon_j=\epsilon_{j+i-i^{\pr}}^{\pr}$. The following theorem is \cite[Theorem 4.4]{reading:2015noncrossing}, which is a translation of \cite{reading:2004lattice}.

\begin{theorem}\label{thm_forcing_arcs}
Given arcs $\alpha$ and $\beta$, $\alpha$ forces $\beta$ if and only if $\alpha$ is a subarc of $\beta$.
\end{theorem}

We say that a lattice congruence $\Theta$ \emph{contracts} an arc if $\Theta^{\alpha} \le \Theta$.
At times we say that $\Theta$ contracts a pair $x\covered y$, when we mean $\Theta^{\alpha(x,y)} \le \Theta$.
Equivalently, $x\equiv_\Theta y$.
Similarly, $\alpha$ is \emph{uncontracted} if $\Theta^{\alpha}\not\subseteq \Theta$.
In this case each pair $x\covered y$ with $\alpha(x,y)=\alpha$ belongs to a distinct $\Theta$-class.

\begin{figure}
    
  \centering
  \includegraphics{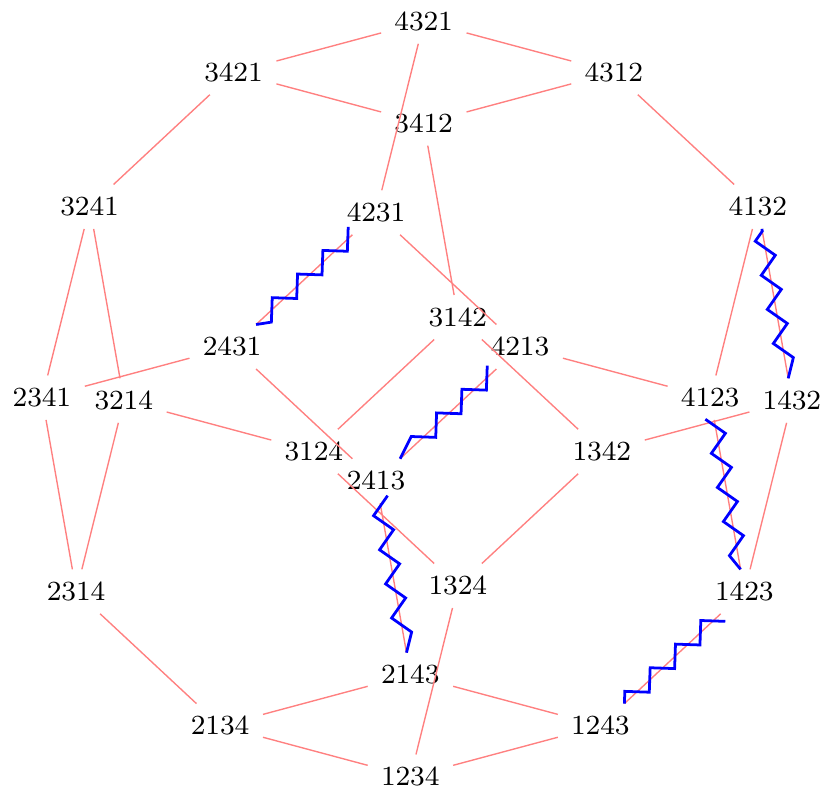}
  \caption{\label{fig_cong}A lattice congruence of the weak order}
  
\end{figure}

\begin{example}
  The weak order on $\Sfrak_4$ is shown in Figure~\ref{fig_cong}. Permutations connected by blue zigzags are equivalence classes of the lattice congruence that contracts the arcs $(2,4,(+)),\ (1,4,(+,+))$ and $(1,4,(-,+))$. The first arc is a subarc of the latter two, so the congruence is the join-irreducible $\Theta^{(2,4,(+))}$.  
\end{example}

The following is \cite[Corollary~4.5]{reading:2015noncrossing}.
\begin{corollary}\label{uncontracted}
A set $U$ of arcs is the set of arcs that are uncontracted by some lattice congruence $\Theta$ if and only if $U$ is closed under taking subarcs.
\end{corollary}

\begin{example}\label{subgraphs}
Let $V\subset [n]$, and consider the the map $\rho: \Sfrak_n\to \Sfrak_{V}$ which sends the permutation $w=w_1\ldots w_n$ to the subword of $w$ in which we delete each $w_i\not\in V$.
Let $\Theta$ denote the smallest (or most refined) lattice congruence on $\Sfrak_n$ in which $\rho(x)= \rho(y)$ implies that $x\equiv _\Theta y$.
We claim that $\rho$ is a lattice map if and only if $V$ is an interval \cite[Example~2.2]{reading:2017homomorphisms}.

First assume that $\rho$ is a lattice map, or, equivalently, the classes of $\Theta$ are precisely the fibers of $\rho$.
Then $\alpha(x,y)$ is uncontracted by $\Theta$ whenever $\rho(x)\ne \rho(y)$.
This happens whenever $x=(i,k)y$ and $i,k\in V$.
Thus, the set of arcs uncontracted by $\Theta$ is $\{(i,k,\epsilon): i,k\in V\}$.
Since this set must be closed under taking subarcs, it follows that $V$ is an interval.

Conversely, if $V$ is an interval then the set $U=\{(i,k,\epsilon): i,k\in V\}$ is the set of uncontracted arcs for some lattice congruence $\Theta'$ (because this set is closed under taking subarcs).
For each $x\covered y$, we have $\alpha(x,y)\notin U$ if and only if $\rho(x)=\rho(y)$.
Thus $\Theta'= \Theta$, and the equivalence classes of $\Theta$ are precisely the fibers of $\rho$.
\end{example}

\subsection{Map from permutations to $G$-forests}
Recall that for any graph $G$ with vertex set $[n]$, and permutation $w=w_1\ldots w_n$, we have $\Psi_G(w) = \{X_1,\ldots, X_n\}$ where $X_i$ is the largest tube in the subset $\{w_1,\ldots,w_j\}$ containing $w_j$.
That is, $X_j$ is the set of vertices of the connected component of $G|_{ \{w_1,\ldots,w_j\}}$ containing $w_j$.
Next, we recursively describe the surjection $\Psi_G: \Sfrak_n\to L_G$ as a map onto the set of $G$-trees. 

Given a connected graph $G$ with vertex set $[n]$ and permutation $w=w_1\ldots w_n$ we recursively construct a $G$-tree $\Psi_G(T)$ as follows:
Let $w_n$ be the root of $T$.
Let $G_1,\ldots, G_r$ be the connected components of the subgraph induced by $\{w_1,\ldots, w_{n-1}\}$.
Restricting $w_1\ldots w_{n-1}$ to each component $G_i$ gives a subword of $w$.
We apply the construction to each subword to obtain subtrees $T_1,\ldots, T_r$.
Finally, we attach each subtree to the root $w_n$.
The next proposition follows from \cite[Corollary~3.9]{postnikov.reiner.williams:2008faces}.
\begin{proposition}\label{linear extensions}
The fiber $\Psi_G^{-1}(T)\subseteq \Sfrak_n$ is the set of linear extensions of~$T$.
\end{proposition}
The authors of \cite{postnikov.reiner.williams:2008faces} define a special section of the map $\Psi_G$, whose image we describe below.
See \cite[Definition~8.7]{postnikov.reiner.williams:2008faces} and Proposition~\ref{prw 8.9}.
\begin{definition}\label{b-permutations}
\normalfont
Let $G$ be a graph with vertex set $[n]$.
A permutation $w$ in $\Sfrak_n$ is a $G$-permutation provided that $w_i$ and $\max\{w_1,\ldots, w_i\}$ lie in the same connected component of $G|_{\{w_1,\ldots, w_i\}}$.
\end{definition}

The following lemma is \cite[Proposition~8.10]{postnikov.reiner.williams:2008faces}.
\begin{lemma}\label{lex smallest}
Let $T$ be a $G$-forest and let $w\in \Phi_G^{-1}(T)$.
Then $w$ is a $G$-permutation if and only if it is the lexicographically minimal linear extension of $T$.
\end{lemma}


Let $V$ be a subset of the vertex set $[n]$, and let $G'$ be the subgraph of $G$ induced by $V$.
We write $\rho_{G'}: L_G\to L_{G'}$ for the map which takes a maximal tubing $\Xcal$ to $\Xcal|_{V}$.
Similarly, write $\rho_{V}: \Sfrak_n\to \Sfrak_{V}$ for the map which sends a permutation $w=w_1w_2\ldots w_n$ to the subword in which we delete each $w_i\notin V$ (without changing the order of the remaining entries).

\begin{lemma}\label{lem: interval lemma}
Let $G$ be a graph with vertex set $[n]$.
If $\Psi_G:\Sfrak_n \to L_G$ is a lattice map, then for each $V\subseteq[n]$ and induced subgraph $G'=G|_{V}$, the map $\Psi_{\std(G')}: \Sfrak_{\std(V)} \to L_{\std(G')}$ is a lattice map.
\end{lemma}
\begin{proof}
Let $V=a_1<a_2<\cdots<a_r$ and $[n]\setminus V = b_1<b_2<\cdots<b_s$, where $r+s=n$.
We consider the interval $I$ in the weak order on $\Sfrak_n$ whose elements consist of all of the permutations on $V$ followed by the fixed permutation $b_1b_2\ldots b_s$.
Observe that for each $\Xcal\in \Psi_G(I)$, the set $V$ is an ideal.
Indeed, let $w=w_1\ldots w_rb_1\ldots b_s$ be a permutation in $I$ and let $\Psi_G(w)=\Xcal=\{X_1\ldots, X_n\}$.
For each connected component $H$ of $G'$ there is a largest integer $j$ in $[r]$ such that $w_j\in H$.
Recall that $X_j$ is the set of vertices of the connected component of $G|_{\{w_1,\ldots w_j\}}$ that contains~$w_j$.
Thus $X_j=H$.

We claim that the following diagram commutes.

\begin{center}
\begin{tikzpicture}[scale=0.75]
  \node (A) at (0,0) {$I$}; 
  \node(B) at (3,0) {$\Psi_G(I)$};
  \node (C) at (0,-2) {$\Sfrak_{V}$}; 
  \node (D) at (3,-2) {$L_{G'}$};
\node (E) at (-0.25,-3.2) {$\Sfrak_{\std(V)}$};
\node (F) at (3.35, -3.2) {$L_{\std(G')}$};
  \node at (1.65,.25) {\scriptsize{$\Psi_G$}};
   \node at (-.35,-1) {\scriptsize{$\rho_{V}$}};
    \node at (3.35,-1) {\scriptsize{$\rho_{G'}$}};
  \node at (1.65,-1.75) {\scriptsize{$\Psi_{G'}$}};
   \node at (1.6,-2.85) {\scriptsize{$\Psi_{\std(G')}$}};
  \draw[->>] (A.east)--(B.west);
  \draw[->] (A.south)--(C.north);
  \draw[->>] (C.east)--(D.west);
  \draw[->] (B.south)--(D.north);
  \draw[-] (3.,-3) -- (3, -2.5);
  \draw[-] (3.1,-3) -- (3.1, -2.5);
    \draw[-] (-.05,-3) -- (-.05, -2.5);
  \draw[-] (.05,-3) -- (.05, -2.5);
   \draw[->>] (E.east) -- (F.west);
\end{tikzpicture}
\end{center}


Set $\Psi_G(w):=\Xcal$ and $\Psi_{G'}(\rho_V(w)):=\Ycal$, where $w= w_1\ldots w_rb_1\ldots b_s$ as above.
Observe that $w_1<w_2<\cdots<w_r$ is a linear extension for both $\tau(\Xcal|_V)$ and $\tau(\Ycal)$.
Therefore, $\rho_{G'}(\Psi_{G}(w))=\Xcal|_{V}=\Ycal=\Psi_{G'}(\rho_{V}(w))$, and the diagram commutes.

Next, we check that $\rho_{G'}: \Psi_{G}(I)\to L_{G'}$ is a poset isomorphism.
Because $\rho_{V}:I\to \Sfrak_V$ is a poset isomorphism, it follows that $\rho_{G'}$ is surjective.
Suppose that $\Xcal, \Ycal\in \Psi_G(I)$ with $\Xcal|_V =\Zcal= \Ycal|_V.$
We will argue that $\tau(\Xcal)=T_{\Xcal}$ and $\tau(\Ycal)= T_{\Ycal}$ are equal.

The only possible difference between $T_\Xcal$ and $T_\Ycal$ must occur among the elements of $V$.
(Each $i,k\notin V$ that are in the same connected component of $G|_{[n]\setminus V}$ are linearly ordered by $b_1<b_2<\cdots<b_s$.)
We write $<_\Xcal$ for the order relation in $T_\Xcal$ and similarly $<_\Ycal$ for the relation in $T_\Ycal$.
Assume that $i<_\Xcal k$ and $i\not<_\Ycal k$ for some $i,k\in V$.
Observe that the pair must be incomparable in the $G'$-forest $\tau(\Zcal)$.
Thus, $\{j\in [n]: j\le_\Xcal k\}\cap V$ is not a tube.
But since $V$ is an ideal in $\Xcal$, we have $\{j\in [n]: j\le_\Xcal k\}\cap V = \{j\in [n]: j\le_\Xcal k\}$.
The latter is clearly a tube.
By this contradiction, we conclude that $\Xcal= \Ycal$, as desired.

\end{proof}

\section{Lattices of maximal tubings}\label{sec_lattice}

\subsection{Right-filled graphs}
We say that a graph $G$ with vertex set $[n]$ and edge set $E$ is \emph{right-filled} provided that the following implication holds:

\begin{equation*}\label{right filled}
\text{If $\{i,k\}\in E$ then $\{j,k\}$ also belongs to $E$ for each $1\le i<j<k\le n$.}\tag{RF}
\end{equation*}

Dually, we say that $G$ is \emph{left-filled} provided that:
\begin{equation*}\label{left filled}
\text{If $\{i,k\}\in E$ then $\{i,j\}$ also belongs to $E$ for each $1\le i<j<k\le n$.}\tag{LF}
\end{equation*}

The goal of this section is two-fold:
First we show that if $G$ is right-filled, then the subposet of the weak order induced by the set of $G$-permutations in~$\Sfrak_n$ is a lattice.
In fact, we show that this subposet is a meet semilattice quotient of the weak order.
(See Corollary~\ref{g-perm lattice cong}.)
Second, we prove that $L_G$ is isomorphic to the subposet of the weak order induced by the set of $G$-permutations in~$\Sfrak_n$.
Hence, $L_G$ is a lattice. 
(See Theorem~\ref{inversion order}.)
\begin{remark}\label{rmk: dualizing and left-filled}
Recall that $G^*$ is the graph obtained from $G$ by swapping the labels $i$ and $n+1-i$ for all $i\in [n]$.
Observe that $G$ is right-filled if and only if $G^*$ is left-filled.
Lemma~\ref{graph duality} says that $L_{G^*} \cong {L_G}^*$, thus we obtain dual versions of  Corollary~\ref{g-perm lattice cong} and Theorem~\ref{inversion order} when $G$ is left-filled.
Some care is required.
In particular, we note that for left-filled graphs, $L_G$ is \emph{not} isomorphic to the subposet induced by the set of $G$-permutations.
\end{remark}

\begin{proposition}\label{connected}
Suppose that $G$ is a right-filled graph with vertex set $[n]$ and connected components $G_i=(V_i, E_i)$ where $i\in [s]$ and $s\ge 2$.
If $\Psi_i: \Sfrak_{V_i} \to L_{G_i}$ is a lattice map for each $i$, then $\Psi_G:\Sfrak_n\to L_G$ is a lattice map.
\end{proposition}
\begin{proof}
We claim that each $V_i$ is an interval.
Write $m_i$ for $\min(V_i)$ and $M_i$ for $\max(V_i)$.
Observe that each geodesic $M_i=q_0, q_1, \ldots, q_k=m_i$ in the graph $G$ monotonically decreases.
That is, $q_0>q_1>\ldots>q_k$.
Indeed, if there exists $q_r>q_{r+1}<q_{r+2}$ then the~\ref{right filled} property implies that $q_r$ and $q_{r+2}$ are adjacent.
Applying the~\ref{right filled} property again, each closed interval $[q_r, q_{r+1}]\subseteq V_i$.
Thus, $V_i$ is an interval.

Observe that the following the diagram commutes.
By Lemma~\ref{lem_decomposition}, the vertical map from $L_G$ to $\prod_{i=1}^sL_{G_i}$ is an isomorphism.
The vertical map from $\Sfrak_n$ onto $\rho:=\prod_{i=1}^s\Sfrak_{V_i}$, where $\rho_{V_i}$ is the restriction map from Example~\ref{subgraphs}.
Since each $V_i$ is an interval, $\rho$ is a lattice map.
The statement of the proposition now follows.
\begin{center}
\begin{tikzpicture}[scale=0.75]
  \node (A) at (0,0) {$\Sfrak_n$}; 
  \node(B) at (5,0) {$L_G$};
  \node (C) at (0,-2) {$\prod_{i=1}^s\Sfrak_{V_i}$}; 
  \node (D) at (5,-2) {$\prod_{i=1}^s L_{G_i}$};
  \node at (2.75,.35) {\scriptsize{$\Psi_G$}};
  \node at (2.5,-1.7) {\scriptsize{$\prod_{i=1}^s\Psi_{G_{i}}$}};
  \draw[->>] (A.east)--(B.west);
  \draw[->>] (A.south)--(C.north);
  \draw[->>] (C.east)--(D.west);
  \draw[->] (B.south)--(D.north);
\end{tikzpicture}
\end{center}

\end{proof}
With Proposition~\ref{connected} in hand, we assume throughout that $G$ is connected.
We realize $L_G$ as a poset on the set of $G$-trees, where $T\le T'$ if and only if $\chi(T)\le \chi(T')$, where $\chi$ is the bijection $T\mapsto \{x_\down: x\in [n]\}$ from Theorem~\ref{G-trees}.

\begin{lemma}\label{lem: child relations}
Let $G$ be a left or right-filled graph, and let $T\in L_G$.
If $x_1$ and $x_2$ are incomparable in $T$, 
then there does not exist any triple $i<j<k$ such that $i$ and $k$ belong to ${x_1}_\down$ and $j\in {x_2}_\down$.
\end{lemma}
\begin{proof}
Consider the set of pairs $i<k$ in ${x_1}_\down$ such that $i<k-1$.
Because ${x_1}_\down$ is a tube, there is a path $i=q_0,\ldots, q_m=k$ in $G$ such that each $q_l$ belongs to ${x_1}_\down$.
Choose such a path so that $m$ is minimal.

We argue by induction on $m$ that there exists no vertex $j$ in ${x_2}_\down$ satisfying $i<j<k$.
Observe that if $j\in {x_2}_\down$ then neither $\{i,j\}$ nor $\{j,k\}$ are edges in $G$.
Thus the base case holds because $G$ is either right-filled or left-filled.

Now assume that $m>1$, let $j\in \{i+1,\ldots, k-1\}$, and for the moment assume that $G$ is right-filled.
Consider $q_{m-1}$.
If $q_{m-1}<j$, then the \ref{right filled}-property implies that $j$ and $k$ are adjacent.
Hence $j\not \in {x_2}_\down$.
If $j<q_{m-1}$ then we have $i<j<q_{m-1}$, and $i$ and $q_{m-1}$ are connected by a path of length~$m-1$.
By induction $j\not \in {x_2}_\down$, and the statement follows.

If $G$ is left-filled the proof is similar, except that we compare $j$ with $q_2$ instead of $q_{m-1}$.
\end{proof}

\begin{proposition}\label{prop: child relations}
Let $G$ be a left or right-filled graph, and let $T\in L_G$.
Suppose that $x_1$ and $x_2$ are incomparable in $T$ and that $x_1<x_2$ as integers.
Then each element in ${x_1}_\down$ is smaller than each element in ${x_2}_\down$ (as integers).
\end{proposition}
\begin{proof}
Set $i:= \max \{a\in {x_2}_{\down}: \text{there exists } b\in {x_1}_\down\text{ with }a< b\}$.
So, there is some $j\in {x_1}_\down$ such that $i<j$.

Assume that $i$ is the largest element in ${x_2}_\down$.
Thus, $x_1<i<j$
(where we have the first inequality because $x_1<x_2$ and $x_2<i$.)
Since $x_1$ and $j$ both belong to ${x_1}_\down$, we have a contradiction to Lemma~\ref{lem: child relations}.

So, there exists some number $k$ in ${x_2}_\down$ with $k>i$, and the maximality of $i$ implies that $k\not<j$ for any $j\in {x_1}_\down$.
Then the triple $i<j<k$ satisfies: $i,k$ both in ${x_2}_{\down}$ and $j\in {x_1}_\down$.
That is a contradiction to Lemma~\ref{lem: child relations} again.
(Note that the roles of $x_1$ and $x_2$ are symmetric in Lemma~\ref{lem: child relations}.)
The proposition follows.
\end{proof}

Below we recursively construct a special linear extension $\sigma(T)$ for $T\in L_G$.
First, if $T$ has a root $x$ then we remove it.
Let $C_1,\ldots,C_r$ be the connected components of $T\setminus \{x\}$.
We index the connected components so that each element of $C_i$ is less than each element of $C_j$ (as integers) whenever $i<j$.
Next, we apply the construction to each component to obtain $v_{{C_i}_1}\ldots v_{{C_i}_s}=\sigma(C_i)$ for $i\in [r]$.
Finally, we concatenate the words $\sigma(C_1)\ldots \sigma(C_r)$, ending with the root~$x$ (if there is one).
Observe that $\sigma(T)$ is the lexicographically minimal linear extension of the $G$-tree $T$.
The next proposition follows from Lemma~\ref{lex smallest} (see also \cite[Proposition~8.10]{postnikov.reiner.williams:2008faces}).
\begin{proposition}\label{prw 8.9}
The image $\sigma(L_G)$ is the equal to the set of $G$-permutations in $\Sfrak_n$.
Moreover, the map $\Psi_G$ induces a bijection between $G$-permutations and $G$-trees, and $\sigma: L_G\to \Sfrak_n$ is a section of the map $\Psi_G$.
\end{proposition}

A pair of numbers $(i,j)$ is an \emph{inversion} of a $G$-tree $T$ if $i<j$ and $j<_T i$.
For example, a descent of $T$ is an inversion such that $i$ covers $j$ in $T$.
A pair $(i,j)$ is a non-inversion if $i<j$ and $i<_T j$. 
(Pairs $i$ and $j$ which are incomparable in $T$ are neither inversions nor non-inversions.)
Write $\inv(T)$ for the set of all inversions of $T$ and $\inv^\wedge(T)$ for the set of noninversions.
The next lemma follows immediately from Proposition~\ref{prop: child relations} and the construction of $\sigma(T)$.
The second item of the statement also follows from \cite[Proposition~9.5]{postnikov.reiner.williams:2008faces}. 

\begin{lemma}\label{lem: pidown}
Let $G$ be left or right-filled graph with vertex set $[n]$.
Suppose that $x$ and $x'$ are incomparable in $T$ and $x$ precedes $x'$ in the linear extension $\sigma(T)$.
Then $x$ is less than $x'$ as integers.
In particular:
\begin{itemize}
\item the inversion set of $T$ is equal to the inversion set of $\sigma(T)$;
\item the descent set of $T$ is equal to the descent set of $\sigma(T)$
\end{itemize} 
\end{lemma}

\begin{remark}\label{rmk: dual section}
Dually we recursively construct a (lexicographically) largest linear extension $\sigma^*(T)$ as follows:
As before $C_1,\ldots, C_r$ are the connected components of $T\setminus \{x\}$ (if $T$ has root $x$) or $T$ (if $T$ does not have a root), indexed so that each element in $C_i$ is less than each element in $C_j$ if $i<j$.
Apply the construction $\sigma^*(C_i)$ to each connected component.
Concatenate the words: $\sigma^*(C_r)\ldots \sigma^*(C_1)$, and finally end with the root $x$.
Indeed, if $G$ is either left or right filled, then $\sigma^*(T)$ is the unique largest element of the fiber $\Psi_G^{-1}(T)$.
\end{remark}


\begin{proposition}\label{inversions}
Suppose that $G$ is a right-filled graph with vertex set $[n]$.
If $w\le w'$ in the weak order on $\Sfrak_n$ then $\inv(\Psi_G(w))\subseteq \inv(\Psi_G(w'))$.
\end{proposition}
\begin{proof}
Write $T$ for $\Psi_G(w)$ and $T'$ for $\Psi_G(w')$.
Suppose that $(i,k)$ is an inversion in~$T$.
Since $w$ is a linear extension of $T$, we have $(i,k)\in \inv(w)$.
Hence $(i,k)\in \inv(w')$.
If $i$ and $k$ are comparable in $T'$, then $(i,k)\in \inv(T')$, since $w'$ is a linear extension of $T'$.

Because $(i,k)\in \inv(T)$, there is a path $i=q_0,\ldots, q_m=k$ (which we take to have minimal length) in $G$ connecting $i$ to $k$ such that $q_l<_T i$ for each $l\in[m]$.
We prove, by induction on $m$, that $i$ and $k$ are comparable in $T'$.
In the base case $i$ and $k$ are adjacent in~$G$, and the claim is immediate.

Assume $m>1$ (so, in particular, $i$ and $k$ are \textit{not} adjacent in $G$).
We make two easy observations:
First, because $G$ is right-filled, $q_{m-1}>i$.
(Indeed, if $q_{m-1}<i<k$ then $G$ must have the edge $\{i,k\}$, contrary to our assumption that $i$ and $k$ are not adjacent.)
Thus, $(i,q_{m-1})\in \inv(T)$.
By induction, $i$ and $q_{m-1}$ are comparable in $T'$.
Thus, $q_{m-1}<_{T'} i$.
Second, because $q_{m-1}$ is adjacent to $k$, they are also comparable in~$T'$.

If $k<_{T'} q_{m-1}$ then we are done by transitivity.
On the other hand, if $q_{m-1} <_{T'} k$ then Lemma~\ref{parent} implies that $k$ and $i$ are comparable in $T'$. 
\end{proof}
We obtain the following corollary.
\begin{corollary}\label{g-perm lattice cong}
Let $G$ be a right-filled graph with vertex set $[n]$ and let $T\in L_G$.
Then the equivalence relation $\Theta_G$ induced by the fibers of $\Psi_G$ satisfies:
\begin{enumerate}
\item  The $\Theta_G$-class $\Psi_G^{-1}(T)$ has a smallest element in the weak order, namely the $G$-permutation $v$ in $\Psi_G^{-1}(T)$;
\item the map $\pidown^G:\Sfrak_n\to \Sfrak_n$ which sends $w$ to the unique $G$-permutation in its $\Theta_G$-class is order preserving.
\end{enumerate}
Thus, the subposet of the weak order on $\Sfrak_n$ induced by the set of $G$-permutations is a meet-semilattice quotient of $\Sfrak_n$.
In particular, the subposet induced by the set of $G$-permutations is a lattice.
\end{corollary}
\begin{proof}
Suppose that $w\in \Psi^{-1}(T)$.
Since $w$ is a linear extension of $T$, $\inv(T)\subseteq \inv(w')$.
Since $\inv(T)=\inv(v)$, we conclude that $v\le w$.
Thus, $v$ is the unique minimal element of the fiber $\Psi^{-1}(T)$.

Suppose that $w\le w'$ in the weak order on $\Sfrak_n$.
Then Proposition~\ref{inversions} says that $\inv(\Psi_G(w))\subseteq \inv(\Psi_G(w'))$.
For each $u\in \Sfrak_n$, $\inv(\pidown^G(u)) = \inv(\Psi_G(u))$, by Lemma~\ref{lem: pidown}.
Thus, $\inv(\pidown^G(w))\subseteq \inv((\pidown^G(w'))$.

The remaining statements of the corollary follow immediately from Proposition~\ref{meet cong}.
\end{proof}

We are now prepared to state the main theorem of this section.
\begin{theorem}\label{inversion order}
Suppose that $G$ is right-filled and $T$ and $T'$ belong to $L_G$.
Then:
\begin{enumerate}
\item $T\le T'$ in $L_G$ if and only if $\inv(T)\subseteq \inv(T')$.
\item The poset of maximal tubings $L_G$ is isomorphic the subposet of the weak order induced by the set of $G$-permutations in $\Sfrak_n$.
In particular, $L_G$ is a lattice.
\item $\Psi_G: \Sfrak_n\to L_G$ is meet semilattice map.
That is, for all $w, w'\in \Sfrak_n$ we have \[\Psi_G(w\meet w') = \Psi_G(w)\meet \Psi_G(w').\]
\end{enumerate}
\end{theorem}

\begin{lemma}\label{descent helper}
Suppose that $G$ is a right-filled graph with vertex set $[n]$, let $a\in [n]$, and let $T'\in L_G$.
Let the disjoint union $C_1\cup C_2\cup\cdots  \cup C_k$ of tubes denote the ideal $a_\down \setminus \{a\}$ in $T$, indexed so that each element of $C_i$ is less than each element in $C_j$ (as integers) whenever $i<j$.
\begin{enumerate}
\item If $(a,x)\in\inv(T)$, then $x$ belongs to~$C_k$.
\item If $(a,x)$ is a descent of $T'$, then swapping $a$ and $x$ we obtain $T$ which satisfies: $$\inv(T)\subseteq \inv(T').$$
\end{enumerate}
\end{lemma}
\begin{proof}
By Proposition~\ref{prop: child relations}, the tubes $C_1,\ldots, C_k$ can be indexed as described in the statement of the lemma.
Suppose there exists $i<k$ and $x\in C_i$ such that $x>a$.
Let $y$ be any element of $C_{i+1}$ that is adjacent to $a$ in $G$.
Because $a<x<y$ and $G$ is right-filled, $x$ and $y$ are adjacent in $G$.
That is a contradiction.

Suppose that $(p,q)\in \inv(T)$.
Hence $p<q$ as integers and $q\in \{y\in [n]: y\le_{T} p\}$.
We must show that $q\in \{y\in [n]: y\le_{T'} p\}$.
If $p$ is not equal to $a$ or $x$ then the statement follows from the fact that $\{y\in [n]: y\le_{T'} p\}=\{y\in [n]: y\le_{T} p\}$.
If $p=a$ then Proposition~\ref{cover relations} implies that $\{y\in [n]: y\le_{T} a\}\subseteq\{y\in [n]: y\le_{T'} a\}$.
Thus $(p,q)\in \inv(T')$.
Assume that $p=x$, so that we have $a<x<q$, ordered as integers.
The first statement of the lemma implies that $q<_{T'} x$ (because $C_k = \{y\in [n]: y\le_{T'} x\}$).
Hence $(p,q)\in \inv(T')$.
\end{proof}

The next lemma is the $G$-tree analog to Lemma~\ref{descents and inversions} (which characterizes covering relations in the weak order on $\Sfrak_n$).

\begin{lemma}\label{lem: inversion covers}
Let $G$ be a right-filled graph with vertex set $[n]$.
Suppose that $T$ and $T'$ are in $L_G$ such that $\inv(T)\subsetneq \inv(T')$.
Then there exists $T''$ such that $T'\covers T''$ and $\inv(T)\subseteq \inv(T'')\subset \inv(T')$.
\end{lemma}
\begin{proof}
We claim that there exists some descent $(i,k)$ of $T'$ that is not an inversion~$T$.
The claim follows from Lemma~\ref{lem: pidown}.
Indeed, write $v$ for $\sigma(T)$ and $v'$ for $\sigma(T')$.
Lemma~\ref{lem: pidown} implies that $v'>v$ in the weak order.
By Lemma~\ref{descents and inversions}, there is some descent $(i,k)$ of $v'$ which is not an inversion of $v$.
Since $\inv(v)=\inv(T)$, $\inv(v')=\inv(T')$, and $\des(v')=\des(T')$ the claim follows.
Let $T''\covered T'$ via this $(i,k)$ descent.

Next, we apply Proposition~\ref{cover relations} to the covering relation $T'\covers T''$.
As in the notation of that proposition, we interpret $x_\down$ as the principal order ideal in $T'$.
We will continue to do so for the remainder of the proof.
Write the ideal $\{x: x<_{T'} k\}$ as the disjoint union of tube $Y_1\cup \cdots \cup Y_t$.
Proposition~\ref{cover relations} says that, $\chi(T'')$ is equal to$$\chi(T')\setminus \{k_\down\} \cup \left\{ i_\down \setminus \left(\{k\} \cup \bigcup Y_{a_j}\right) \right\},$$
where the union $\bigcup Y_{a_j}$ is over all $Y_{a_j}\in \{Y_1,\ldots, Y_t\}$ such that $\{i\}\cup Y_{a_j}$ not a tube.
We write $B$ for the set $\{b\in\bigcup Y_{a_j}: (i,b)\in \inv(T')\}.$
It follows that $$\inv(T')\setminus (\{(i,k)\}\cup \{(i,b): b\in B\})= \inv(T'').$$

Let $C$ be the set of $c\in \bigcup Y_{a_j}$ such that $(i,c)\in \inv(T)$.
(As above, each $Y_{a_j}$ satisfies: $Y_{a_j}\cup\{i\}$ is not a tube; so in particular, no element $c\in C$ is adjacent to $i$.)
To complete the proof, we argue that $C$ is empty.
Suppose not, and choose $c\in C$ so that there is a path $i=q_0, q_1,\ldots, q_m=c$ with $q_p \le_T i$ and $q_p\not \in C$ for each $p\in [m-1]$.

Consider $q_{m-1}$.
On the one hand, if $q_{m-1}<i$ (as integers) then the \ref{right filled}-property implies that $i$ and $c$ are adjacent.
But no element in $C$ is adjacent to $i$.
So we have a contradiction.

On the other hand, if $q_{m-1}>i$ then $(i,q_{m-1})$ is an inversion of $T$.
We will argue that $q_{m-1}$ must belong to $C$, and conclude a contradiction.
Since $\inv(T)\subset \inv(T')$ have have $(i,q_{m-1})\in \inv(T')$.
Thus $q_{m-1}$ is in the ideal $\{x: x<_{T'} i\}$, which we write as a disjoint union of tubes $X_1\cup X_2\cup \cdots \cup X_r$.
Because $i\covers_{T'} k$ and $(i,k)$ is a descent in $T'$, we have $k_{\down}=X_r$ by Lemma~\ref{descent helper}.
Since $q_{m-1}$ is adjacent to $c$, both belong to the same tube in the disjoint union $X_1\cup\cdots \cup X_r$.
Because $c\in Y_{a_j}\subseteq k_\down$, we have $q_{m-1}$ is also in $k_\down$.
Similarly, because $c$ and $q_{m-1}$ are adjacent, they belong to the same tube $Y_{a_j}$ in the ideal $\{x: x<_{T'} k\} = Y_1\cup \cdots \cup Y_t$.
We conclude that $q_{m-1}\in C$, contradicting our choice of~$c$.
Therefore, $\inv(T)\subseteq \inv(T'') \subset \inv(T')$.

\end{proof}

\begin{proof}[Proof of Theorem~\ref{inversion order}]
Lemma~\ref{descent helper} implies that $T\le T'$ then $\inv(T)\subseteq \inv(T')$.
Suppose that $\inv(T)\subseteq \inv(T')$.
We argue that $T\le T'$ by induction on the size of $\inv(T')\setminus \inv(T)$.
Lemma~\ref{lem: inversion covers} says there exists a $G$-tree $T''$ such that $$\text{$T'\covers T''$ and $\inv(T)\subseteq \inv(T'')\subset \inv(T')$}.$$

Lemma~\ref{lem: pidown} implies that $\inv(T)=\inv(T')$ if and only if $T=T'$.
When $\inv(T)$ and $\inv(T')$ differ by one element, we must have that $T=T''$.
When $\inv(T')\setminus \inv(T)$ has $m>1$ elements, the inductive hypothesis implies that $T\le T''\covered T'$, and we are done.
\end{proof}

\subsection{Left-filled graphs}
In this section we prove the analog of Corollary~\ref{g-perm lattice cong} and Theorem~\ref{inversion order} for left-filled graphs.

\begin{corollary}\label{left filled cor}
Let $G$ be a left-filled graph with vertex set $[n]$.
Then $L_G$ is a lattice, and $\Psi_G: \Sfrak_n\to L_G$ is a join semilattice map.
That is, for all $w, w'\in \Sfrak_n$, we have 
\[\Psi_G(w\join w')= \Psi_G(w)\join \Psi_G(w').\]
\end{corollary}
\begin{proof}
Observe that $G^*$ is right-filled.
(Recall that $G^*$ is the graph we obtain by swapping labels $i$ and $n+1-i$ for all $i$.)
Lemma~\ref{graph duality} says that $L_{G^*} \cong L_G^*$ (where $L_G^*$ is the dual of $L_G$, as posets).
Since $L_{G^*}$ is a lattice (by Theorem~\ref{inversion order}), we have $L_G$ is a lattice.

Indeed, Lemma~\ref{graph duality} implies that the following diagram commutes.
\begin{center}
\begin{tikzpicture}[scale=0.75]
  \node (A) at (0,0) {$\Sfrak_n$}; 
  \node(B) at (3,0) {$\Sfrak_n^*$};
  \node (C) at (0,-2) {$L_G$}; 
  \node (D) at (3,-2) {$L_{G^*}$};
  \node at (-.5,-.85) {\scriptsize{$\Psi_G$}};
  \node at (3.5,-.85) {\scriptsize{$\Psi_{G^*}$}};
  \draw[<->] (A.east)--(B.west);
  \draw[->>] (A.south)--(C.north);
  \draw[<->] (C.east)--(D.west);
  \draw[->>] (B.south)--(D.north);
\end{tikzpicture}
\end{center}

The maps in the top and bottom rows of the diagram are essentially the same:
They both swap $i$ and $n+1-i$ for all $i\in [n]$.
Both maps are lattice anti-isomorphisms.
It follows from Theorem~\ref{inversion order} that $\Psi_G$ is a join semilattice map.
\end{proof}

\subsection{Filled graphs and lattice congruences}
We prove our main result (see Theorem~\ref{thm_main_lattice}).
\begin{theorem}\label{main}
Suppose that $G$ is a graph with vertex set $[n]$ and edge set $E$.
Then $\Psi_G: \Sfrak_n\to L_G$ is a lattice quotient map if and only if $G$ is filled.
\end{theorem}
\begin{proof}
If $G$ is filled, then $\Psi_G: \Sfrak_n \to L_G$ preserves the meet operation (by Theorem~\ref{inversion order}) and the join operation (by Corollary~\ref{left filled cor}).
Thus $\Psi_G:\Sfrak_n\to L_G$ is a lattice quotient map.

Assume that $G$ is not filled.
Thus there exists $i<j<k$ such that $\{i,k\}\in E$ but either $\{i,j\}$ or $\{j,k\}$ is not in $E$.
Let $G'$ denote the induced subgraph $G|_{\{i,j,k\}}$.
We check that in all possible cases $\Psi_{\std(G')}:\Sfrak_{3}\to L_{\std(G')}$ is not a lattice map.
By Lemma~\ref{lem: interval lemma}, $\Psi_G: \Sfrak_n\to L_G$ is not a lattice map.

In the first case, assume $\{i,k\}$ and $\{j,k\}$ are edges, but $\{i,j\}$ is not.
Observe that $\Psi_{\std(G')}$ does not preserve the join operation.
On the one hand, $213\join 132 = 321$ in the weak order on $\Sfrak_3$.
Thus,
 \[\Psi_{\std(G')} (213\,\join\, 132) = \Psi_{\std(G')}(321)= \substack{1\\2\\3},\]
 where we write $\Psi_{\std(G')}(321)$ as a $\std(G')$-tree (with the elements ordered vertically).
 On the other hand,
 \[ \Psi_{\std(G')}(213) \join \Psi_{\std(G')}(132)= \substack{3\\[.25em]1\,2} \join\, \substack{2\\3\\1} =\substack{3\\1\,2}\]
The reader can check the computation of the join in $L_{\std(G')}$ with Figure~\ref{fig:L_ex}.
The case in which $\{i,k\}$ and $\{i,j\}$ are edges (but $\{j,k\}$ is not) is proved dually.

Assume that $\{i,k\}$ is an edge and neither $\{j,k\}$ and nor $\{i,j\}$ are edges.
Then, for example, $\Psi_{\std(G')}$ does not preserve the join operation.
Indeed
\[\Psi_{\std(G')}(123)=\Psi_{\std(G')}(213) = \Psi_{\std(G')}(132)\]
is the smallest element in $L_{\std(G')}$.
But $\Psi_{\std(G')}(213\,\join \,132)$ is the biggest element in $L_{\std(G')}$.

We conclude that if $G$ is not filled, then $L_G$ is not a lattice map.
\end{proof}

\subsection{Generators of the congruence $\Theta_G$}
Let $\Theta_G$ be the equivalence relation on $\Sfrak_n$ induced by the fibers of $\Psi_G$.
In light of Theorem~\ref{main}, when $G$ is filled $\Theta_G$ is a lattice congruence on the weak order.
Recall from Section~\ref{subsec_weak_cong} that $\Con(\Sfrak_n)$ is a finite distributive lattice.
We identify each congruence $\Theta$ with the corresponding order ideal of join-irreducible congruences.
The \emph{generators} of a congruence are the maximal elements of this order ideal.
Recall that the join-irreducible congruences of the weak order are given by arcs. (This is Theorem~\ref{thm_weak_arcs}.)

Let $(x,y, +)$ denote the arc with $\epsilon_i = +$ for each $i\in [y-x]$.
Occasionally we call such an arc a \emph{positive arc}.
(Pictorially, this is an arc which does not pass to the right of any point between its endpoints.)
A \emph{minimal non-edge} is a pair $x<y$ such that for each $z\in\{x+1, x+2,\ldots,y-1\}$, $\{x,z\}$ and $\{z,y\}$ are edges in $G$, but $\{x,y\}$ is not an edge.

\begin{theorem}\label{generators}
Suppose that $G$ is a filled graph, and consider the lattice congruence $\Theta_G$ induced by $\Psi_G$.
Then $\Theta_G$ is generated by $$\{(x,y,+): \{x,y\}\text{ is a minimal non-edge of $G$}\}.$$
\end{theorem}

Before we prove Theorem~\ref{generators} we gather some useful facts.
Throughout this section, we write $j$ for a join-irreducible permutation and $j_*$ for the unique element that it covers in the weak order.
Recall that we associate each $j$ with the join-irreducible congruence $\Theta^\alpha$ where $\alpha$ is the arc $\alpha(j_*,j)$.
Conversely, given an arc $\alpha=(x,y,\epsilon)$, the corresponding join irreducible permutation is
\[j_{\alpha}= 12\ldots (x-1) l_1\ldots l_p \,y\, x\, r_1\ldots r_q (y+1) (y+2) \ldots n\]
where $\{l_1<l_2<\ldots < l_p\}$ is the set $\{x':\epsilon_{x'-x} = -\}$ and $\{r_1<r_2<\ldots<r_q\}$ is the set $\{y': \epsilon_{y'-x} = +\}$.
(See Proposition~\ref{arc to perm}.)

We say a join-irreducible element $j$ is \emph{contracted by $\Theta_G$} if $j\equiv_G  j_*$.
The congruence $\Theta^\alpha$ is a generator for $\Theta_G$ if $j_\alpha$ is contracted by $\Theta_G$, and for each subarc $\beta$ of $\alpha$, the corresponding permutation $j_{\beta}$ is \textit{not} contracted.
The next result follows immediately from Corollary~\ref{g-perm lattice cong}.
\begin{proposition}
Let $G$ be a filled graph with vertex set $[n]$ and let $j$ be a join-irreducible permutation in $\Sfrak_n$.
Then $j$ is not contracted by $\Theta_G$ if and only if $j$ is $G$-permutation.
\end{proposition}

\begin{lemma}\label{not contracted}
Let $G$ be a filled graph with vertex set $[n]$ and $1\le x<y\le n$.
Suppose that $\{x,y\}$ is an edge in $G$, with $x<y$.
Then no arc $(x,y,\epsilon)$ is contracted by $\Theta_G$.
\end{lemma}
\begin{proof}
Let $j$ be join-irreducible with unique descent $(x,y)$.
Observe $G|_{[x,y]}$ is a complete graph because $G$ is filled.

Let $r=y-x+1$.
Write $j$ in one-line notation as:
 $$j= j_1\ldots j_n=12\ldots (x-1) j_x\ldots j_{x+r} (y+1) (y+2)\ldots n.$$
We claim that $j_i$ and $\max\{j_1,\ldots j_i\}$ belong to the same connected component of $G|_{\{j_1,\ldots j_i\}}$. 
If  $i\le x$ or $i \ge x+r+1$ then $j_i =  \max\{j_1,\ldots j_i\}$.
So the claim follows.

Suppose that $x<i\le x+r$.
Then $\max\{j_1,\ldots j_i\} = \max\{j_x,\ldots, j_i\}$.
Since $\{j_x,\ldots, j_i\}$ is a subset of $[x,y]$ claim follows.
Therefore $j$ is a $G$-permutation.
\end{proof}

\begin{lemma}\label{+ is contracted}
Suppose that $(x,y)$ is not an edge in $G$.
The arc (x,y,+) is contracted by~$\Theta_G$.
\end{lemma}
\begin{proof}
Let $j_{x,y}$ denote the join-irreducible corresponding to the arc $(x,y,+)$.
We argue that $j_{x,y}$ is contracted by $\Theta_G$.
Since $x$ and $y$ are not adjacent and $G$ is filled, $y$ is not connected to any vertex $x'<x$.
Write $j_{x,y}$ as $$1\,2\,\ldots (x-1) \,y x \, (x+1)\ldots (y-1)\,(y+1)\,\ldots n.$$
Observe that $j_{x,y}$ is not a $G$-permutation because $y=\max\{1,2,\ldots, x,y\}$ is isolated in the subgraph $G|_{\{1,2,\ldots, x,y\}}$.
Thus $j$ is contracted by $\Theta_G$.
\end{proof}

\begin{lemma}\label{minimal nonedges}
Suppose that $\{x,y\}$ is a minimal non-edge of $G$.
Let $\alpha$ be the arc $(x,y,\epsilon)$.
If $\epsilon\ne +$ then $\alpha$ is not contracted by $\Theta_G$.
\end{lemma}
\begin{proof}
Let $j$ be a join-irreducible corresponding to $\alpha$.
Let $r=x-y+1$.
Write $j$ in one-line notation as $j_1\ldots j_n$.
Observe that $$j=12\ldots (x-1) j_x\ldots j_{x+r} (y+1) (y+2)\ldots n.$$

We claim that $j_i$ and $\max\{j_1,\ldots j_i\}$ belong to the same connected component of $G|_{\{j_1,\ldots j_i\}}$. 
If  $i\le x$ or $i \ge x+r+1$ then $j_i =  \max\{j_1,\ldots j_i\}$.
So the claim follows.

Suppose that $x<i\le x+r$.
Then $\max\{j_1,\ldots j_i\} = \max\{j_x,\ldots, j_i\}$.
Because $G$ is filled, our hypotheses imply that we have each edge in $\binom{[x,y]}{2}$ except $(x,y)$.
If $\max\{j_x,\ldots, j_i\}$ is not equal to $y$ then $G|_{\{j_x,\ldots, j_i\}}$ is a complete graph.
So the claim follows.

Assume that $\max\{j_x,\ldots, j_i\}=y$.
Because $\alpha\ne (x,y,+)$ we have $x<j_x<y$.
Thus $(x,j_x)$ and $(j_x,y)$ are both edges in $G$.
Therefore $G|_{\{j_x,\ldots, j_i\}}$ is connected.
So the claim follows, and $j$ is $G$-permutation.

\end{proof}

\begin{figure}

  \centering
  \includegraphics{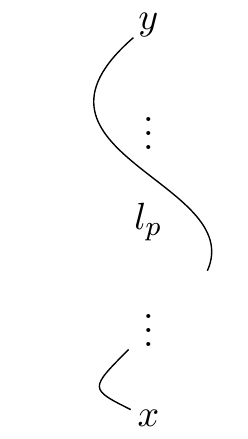}
  \caption{\label{fig:arc2}An arc considered in the proof of Theorem~\ref{generators}}
  
\end{figure}

\begin{proof}[Theorem~\ref{generators}]
Let $\Gcal$ denote the set $\{(x,y,+): (x,y)\text{ is a minimal non-edge}\}$.
Lemma~\ref{+ is contracted} implies that the join-irreducible elements in $\Gcal$ are among the generators of~$\Theta_G$.
To prove the theorem we argue that they are the only generators.

By way of contradiction, assume that $(x,y,\epsilon)$ is a generator of $\Theta_G$ and $(x,y,\epsilon) \notin \Gcal$.
Write $j$ for the corresponding join-irreducible.
By Lemma~\ref{not contracted}, $(x,y)$ is not an edge (because $j$ is contracted by $\Theta_G$).
If $\{x,y\}$ is a minimal non-edge then Lemma~\ref{minimal nonedges} says that $\epsilon=+$, and hence $(x,y,\epsilon)\in \Gcal$.
Thus we may assume that $\{x,y\}$ is not a minimal non-edge, and it has a subarc $\alpha'$ with end points $x'<y'$ which is a minimal non-edge.
Since no subarc of $\alpha$ is contracted, in particular $\alpha'$ is not contracted.
It follows that $\alpha'$ is not a positive arc.
Therefore, $\epsilon \ne +$.

To obtain a contradiction we argue that $j$ is a $G$-permutation.
Write $j$ as $$12\ldots (x-1) l_1\ldots l_p \,y\,x\, r_1 \ldots r_q (y+1) \ldots n$$ where $\{l_1<\ldots< l_p, r_1< \ldots< r_q\}= \{x+1, \ldots y-1\}$.
Therefore, we get $j_i = \max\{j_1,\dots j_i\}$ for $j_i \in \{1,2,\ldots l_p, (y+1), \ldots, n\}$.

We claim that $r_i$ is in the same connected component as $y$ in the subgraph induced by $\{1,2,\ldots,x,y,r_1,\ldots r_i\}$.
Let $j'$ be the join-irreducible whose corresponding arc is the subarc of $(x,y,\epsilon)$ with endpoints $x'=r_i<y$.
As above, write 
$$j'=1\,2\ldots (x'-1) \,l'_1\ldots l'_{p'} \,y\,x'\, r'_1 \ldots r'_{q'}\, (y+1) \ldots n.$$
Observe that $\{1,2,\ldots, (x'-1), l'_1,\ldots, l'_{p'} ,y,x'\} = \{1,2,\ldots,x,y,r_1,\ldots r_i\}$.
(Each of the entry $l_k$ remains left of the descent $(x' y')$ in $j'$ because $l_k<y$.
Each $r_k$ with $k<i$ is left of the descent $(x',y)$ because $r_k<r_i$.)

Because $j'$ is not contracted, it is a $G$-permutation.
Therefore $y$ and $r_i=x'$ belong to the same connected in the subgraph $G|_{\{1,2,\ldots,x,y,r_1,\ldots r_i\}}$.

Finally we consider $x$ and $y= \max\{1,2,\ldots, (x-1), l_1, \ldots, l_p, x,y\}$ in the subgraph $G|_{\{1,2,\ldots, (x-1), l_1, \ldots, l_p, x,y\}}$.
We will be done if we can show that $x$ and $y$ belong to the same connected component.
We do so by showing, first, that $x$ and $l_p$ belong to the same connected component in this subgraph, and second, that $l_p$ and $y$ belong the same connected component.
(Indeed we will see that $l_p$ and $y$ are adjacent in $G$.)
Observe that $l_p$ exists because $\epsilon \ne +$.

Consider the permutation $$j''=1\,2\ldots (x-1) l_1\ldots l_p \,x\, r_1 \ldots r_q \, y\, (y+1) \ldots n.$$
Observe that this permutation has a unique descent $(l_p,x)$, so it is join-irreducible.
Moreover, the arc corresponding to $j''$ is the subarc of $(x,y,\epsilon)$ with endpoints $x<l_p$.
Hence $j''$ is a $G$-permutation.
Thus $x$ and $l_p$ belong to the same connected component in the subgraph $G|_{\{1,2,\ldots, (x-1), l_1, \ldots, l_p, x\}}$.
So, $x$ and $l_p$ belong to the same connected component in~$G|_{\{1,2,\ldots, (x-1), l_1, \ldots, l_p, x, y\}}$.

Next consider the subarc of $(x,y,\epsilon)$ with endpoints $l_p <y$.
As none of the $l_i$ lie strictly between $l_p$ and $y$, this is a positive arc; see Figure~\ref{fig:arc2}.
Since it is not contracted, we must have $l_p$ and $y$ form an edge (by Lemma~\ref{+ is contracted}).

We conclude that $x$ and $y$ belong to the same connected component in the sugraph induced by $\{1,2,\ldots, (x-1), l_1, \ldots, l_p, x,y\}$.
Thus $j$ is a $G$-permutation.
By this contradiction, we obtain the desired result.
\end{proof}

\section{Algebras and coalgebras of tubings}\label{sec_hopf}

\subsection{The Malvenuto-Reutenauer algebra}\label{subsec_MR}

Fix a field $\Kbb$. For a set $X$, we let $\Kbb[X]$ denote the vector space over $\Kbb$ for which the set $X$ indexes a basis. For $X=\Sfrak_n$, we let $\Kbb[\Sfrak_n]$ have a distinguished basis $\{\Fbb_w:\ w\in\Sfrak_n\}$. The \emph{Malvenuto-Reutenauer} algebra is a Hopf algebra on the graded vector space
\[ \Kbb[\Sfrak_{\infty}]=\bigoplus_{n=0}^{\infty}\Kbb[\Sfrak_n]. \]
If $v=v_1\cdots v_n$ is a permutation of $[n]$ and $m\geq 0$, we define the \emph{shift by $m$} to be the word $v[m]=(v_1+m)(v_2+m)\cdots(v_n+m)$. For basis elements $\Fbb_u\in\Kbb[\Sfrak_m],\ \Fbb_v\in\Kbb[\Sfrak_n]$, the product $\Fbb_u\cdot\Fbb_v$ is the sum of the elements $\Fbb_w$ for which $w$ is a shuffle of $u$ and $v[m]$. For example,
$$\Fbb_{21}\cdot\Fbb_{12}=\Fbb_{2134}+\Fbb_{2314}+\Fbb_{2341}+\Fbb_{3214}+\Fbb_{3241}+\Fbb_{3421}.$$
The coproduct $\Delta(\Fbb_u)\in\Kbb[\Sfrak_{\infty}]\otimes\Kbb[\Sfrak_{\infty}]$ for $u\in\Sfrak_n$ is defined to be
$$\Delta(\Fbb_u)=\sum_{i=0}^n\Fbb_{\std(u_1\cdots u_i)}\otimes\Fbb_{\std(u_{i+1}\cdots u_n)},$$
where $\std(a_1\cdots a_i)$ for a sequence of distinct integers $a_1,\ldots,a_i$ is the element of $\Sfrak_i$ with the same relative order as $a_1\cdots a_i$. For example,
$$\Delta(\Fbb_{3241})=\iota\otimes \Fbb_{3241} + \Fbb_1\otimes \Fbb_{231} + \Fbb_{21}\otimes \Fbb_{21} + \Fbb_{213}\otimes \Fbb_1 + \Fbb_{3241}\otimes\iota.$$
Here, the element $\iota\in\Kbb[\Sfrak_0]$ is the multiplicative identity. The counit $\epsilon:\Kbb[\Sfrak_{\infty}]\ra\Kbb$ is the linear map with $\epsilon(\iota)=1$ and $\epsilon(\Fbb_v)=0$ for $v\in\Sfrak_n,\ n\geq 1$. These operations are compatible in a way that makes $\Kbb[\Sfrak_{\infty}]$ a (connected, graded) bialgebra. This automatically gives the Malvenuto-Reutenauer algebra the structure of a Hopf algebra; that is, it comes with a (unique) antipode $S$. We refer to \cite{grinberg.reiner:2014hopf} for further background on Hopf algebras from a combinatorial perspective.

The Malvenuto-Reutenauer algebra contains the algebra of noncommutative symmetric functions $\NCSym$ as a sub-Hopf algebra. Loday and Ronco \cite{loday.ronco:1998hopf} discovered a Hopf algebra $\Kbb[Y_{\infty}]=\bigoplus\Kbb[Y_n]$ on the vector space spanned by planar binary trees and a sequence of Hopf algebra embeddings
$$\NCSym\hookra\Kbb[Y_{\infty}]\hookra\Kbb[\Sfrak_{\infty}].$$

More generally, we may consider a family of nonempty sets $\{Z_0,Z_1,Z_2,\ldots\}$ with surjections $f_n:\Sfrak_n\thra Z_n$ for each $n\geq 0$. Letting $\{\Pbb_x:\ x\in Z_{\infty}\}$ be a basis for $\Kbb[Z_{\infty}]$, there is a vector space embedding $c:\Kbb[Z_{\infty}]\hookra\Kbb[\Sfrak_{\infty}]$ where
$$c(\Pbb_x)=\sum_{w\in f_n^{-1}(x)} \Fbb_w\ \hspace{3mm}\ \mbox{for }x\in Z_n.$$
We are especially interested in the case where $Z_n$ is the set of vertices of a generalized permutahedron of rank $n$ and $f_n:\Sfrak_n\thra Z_n$ is the canonical map. The main problem is to determine whether $c$ makes $\Kbb[Z_{\infty}]$ into an algebra or a coalgebra, i.e. whether $c(\Pbb_x)\cdot c(\Pbb_y)$ and $\Delta(c(\Pbb_x))$ lie in the image of $c$ for any $x,y\in Z_{\infty}$.

\subsection{Translational families of lattice congruences}\label{subsec_MR_alg}

As usual, we consider the symmetric group $\Sfrak_n$ as a poset under the weak order. When the map $f_n:\Sfrak_n\thra Z_n$ has the structure of a lattice quotient map, there is a generalized permutahedron known as a \emph{quotientope} with vertex set $Z_n$ associated to the map $f_n$ \cite{pilaud.santos:2017quotientopes}. In \cite{reading:2005lattice}, Reading proved that the embedding $c$ associated to a sequence of lattice quotient maps $\{f_n\}_{n\geq 0}$ is an algebra map (resp., coalgebra map) if the family $\{f_n\}$ is translational (resp., insertional). We recall the definition of a translational family in this section and of an insertional family in Section~\ref{subsec_insertional}.

Let $\Theta$ be a lattice congruence of the weak order on $\Sfrak_n$ for some $n$. 
Recall that $\Theta$ contracts a join-irreducible $j$ if $j\equiv j_*\mod \Theta$, where $j\covers j_*$.
Equivalently, for the corresponding arc $\alpha=\alpha(j_*,j)$, we have $\Theta^{\alpha}\leq\Theta$ in the lattice $\Con(L)$.
We abuse notation, and say that $\Theta$ \emph{contracts} the arc $\alpha$ if $\Theta$ contracts $j_{\alpha}$.
(Indeed, $\Theta$ contracts an arc $\alpha$ if and only if there exists a covering relation $u\lessdot w$ such that $\alpha(u,w)=\alpha$ and $u\equiv w\mod \Theta$.) 
In particular, the set of arcs contracted by $\Theta$ correspond to the set of join-irreducible elements of $\Con(L)$ less than or equal to $\Theta$ in $\Con(L)$. By Theorem~\ref{thm_forcing_arcs}, if $\alpha$ is contracted by $\Theta$ and $\alpha$ is a subarc of $\beta$, then $\beta$ is contracted by $\Theta$ as well.

Fix a sequence $\mathbf{\Theta}=\{\Theta_n\}_{n\geq 0}$ where $\Theta_n$ is a lattice congruence of the weak order on $\Sfrak_n$ for each $n\geq 0$. We let $Z_n=\Sfrak_n/\Theta_n$ be the set of equivalence classes modulo $\Theta_n$, and set $Z_{\infty}^{\mathbf{\Theta}}=\{Z_n\}_{n\geq 0}$.

As we consider lattice congruences of the weak order for varying $n$, we may say that $\alpha$ is an \emph{arc on $[n]$} to mean that it is an arc for $\Sfrak_n$. An arc $\alpha=(i,j,\epsilon)$ on $[n]$ is a \emph{translate} of an arc $\beta=(k,l,\epsilon^{\pr})$ on $[m]$ if $j-i=l-k$ and $\epsilon=\epsilon^{\pr}$. The family $\{\Theta_n\}_{n\geq 0}$ is called \emph{translational} if whenever $\Theta_n$ contracts an arc $\alpha$ and when $\beta$ is an arc on $[m]$ that is a translate of $\alpha$, the congruence $\Theta_m$ contracts $\beta$.

The following is equivalent to \cite[Theorem 1.2, Proposition 7.1]{reading:2005lattice}.

\begin{theorem}\label{thm_translational_subalg}
  If $\mathbf{\Theta}=\{\Theta_n\}_{n\geq 0}$ is a translational family, then the map
  $$c:\Kbb[Z_{\infty}^{\mathbf{\Theta}}]\ra\Kbb[\Sfrak_{\infty}]$$
  embeds $\Kbb[Z_{\infty}^{\mathbf{\Theta}}]$ as a subalgebra of $\Kbb[\Sfrak_{\infty}]$.
\end{theorem}

We proved (Theorem~\ref{main}) that the map $\Psi_G:\Sfrak_n\ra L_G$ is a lattice map if and only if $G$ is a filled graph. We determine when a sequence of filled graphs determines a translational family of lattice congruences of the weak order. As before, we will write $(i,j,+)$ to represent the arc $(i,j,(+,\ldots,+))$. An arc of the form $(i,j,+)$ is called a \emph{positive arc}. For nonnegative integers $k,n$, let $H_{k,n}$ be the graph with vertex set $[n]$ such that $\{i,j\}$ is an edge whenever $1\leq i<j\leq n$ and $j-i\leq k$. Clearly, if $k\geq n-1$, then $H_{k,n}$ is the complete graph on $[n]$.

\begin{proposition}\label{prop_translational_char}
  A sequence of filled graphs $\{G_n\}_{n\geq 0}$ determines a translational family $\{\Theta_n\}_{n\geq 0}$ if and only if there exists some $k\in\{0,1,2,\ldots\}\cup\{+\infty\}$ such that $G_n=H_{k,n}$ for all $n$.
\end{proposition}

\begin{proof}
  Let $\{G_n\}_{n\geq 0}$ be a sequence of filled graphs, and let $\Theta_n$ be the lattice congruence induced by $\Sfrak_n\ra L_{G_n}$.

  Suppose the family $\{\Theta_n\}_{n\geq 0}$ is translational. If $\{i,j\}$ is not an edge of $G_n$, then $\Theta_n$ contracts the arc $(i,j,+)$. Being a translational family means that any arc of the form $(i^{\pr},j^{\pr},+)$ is contracted by $\Theta_m$ where $1\leq i^{\pr}<j^{\pr}\leq m$ and $j-i=j^{\pr}-i^{\pr}$. This in turn means that $\{i^{\pr},j^{\pr}\}$ is not an edge of $G_m$. Hence, there must exist some set $S\subseteq\Nbb$ such that for all $i,j,n$ such that $1\leq i<j\leq n$, we have $j-i\in S$ if and only if $\{i,j\}$ is an edge of $G_n$. As the graphs $G_n$ are filled, the set $S$ must either be of the form $S=[k]$ for some $k\in\{0,1,2,\ldots\}$ or $S=\Nbb$.

  Conversely, suppose there exists $k\in\{0,1,2,\ldots\}\cup\{\infty\}$ such that $G_n=H_{k,n}$ for all~$n$. Then $\Theta_n$ is the lattice congruence generated by $\{\Theta^{(i,j,+)}:\ j-i=k+1\}$. Since the generating set is closed under translation, the family $\{\Theta_n\}_{n\geq 0}$ must be translational.
\end{proof}

\begin{remark}
  The lattice congruence $\Theta_n$ in Proposition~\ref{prop_translational_char} resembles the \emph{metasylvester congruence} $\equiv_n^k$, which is the lattice congruence of the weak order on $\Sfrak_n$ generated by relations of the form
  $$UacV_1b_1\cdots V_kb_kW\equiv_n^k UcaV_1b_1\cdots V_kb_kW,$$
  where $a<b_i<c$ holds for all $i\in[k]$. In other words, two letters $a,c$ in a permutation can be swapped if there are $k$ letters on the right with values between $a$ and $c$. Via the dictionary between arcs and join-irreducible lattice congruences, the metasylvester congruence is the most refined congruence that contracts every arc of the form $(a,c,\epsilon)$ where the number of $+$ entries in $\epsilon$ is at least $k$. In contrast, the congruence $\Theta_n$ corresponding to the graph $H_{k,n}$ is generated by $\Theta^{(i,j,+)}$ where $j-i=k+1$, meaning that $\epsilon_1=\cdots=\epsilon_k=+$.

  It is straight-forward to check that the family $\{\equiv_n^k\}_{n\geq 1}$ is both translational and insertional; cf. Section~\ref{subsec_insertional}. Hence, the embedding
  $$c:\Kbb[\Sfrak_{\infty}/\equiv^k]\hookra\Kbb[\Sfrak_{\infty}]$$
 realizes $\Kbb[\Sfrak_{\infty}/\equiv^k]$ as both a subalgebra and a sub-coalgebra of the Malvenuto-Reutenauer Hopf algebra. It follows that $\Kbb[\Sfrak_{\infty}/\equiv^k]$ inherits the structure of a bialgebra. It is known that an antipode is inherited as well, giving it the structure of a sub-Hopf algebra. Pilaud \cite{pilaud:2018brick} interpreted this Hopf algebra in terms of the vertices of a family of \emph{brick polytopes}, which are a different class of generalized permutahedra from those we consider in this paper.
\end{remark}

\subsection{Tubing algebras}\label{subsec_hopf_algebra}

We begin this section by recalling the tubing algebra defined by Ronco \cite{ronco:2012tamari}.

For $I\subseteq\Nbb,\ n\geq 0$, let $I+n:=\{i+n\ |\ i\in I\}$. In particular, $[m]+n=\{n+1,n+2,\ldots,n+m\}$. If $\Xcal$ is a tubing, we let $\Xcal+n:=\{I+n\ |\ I\in\Xcal\}$.

Consider a family of graphs $\Gcal=\bigsqcup_{n\geq 0}\Gcal_n$ where $\Gcal_n$ is a finite collection of graphs with vertex set $[n]$. We allow $\Gcal_n$ to contain multiple copies of the same graph, and for the purposes of defining the tubing algebra, it will be important to be able to distinguish between multiple copies of the same graph. This could be done by defining $\Gcal$ as a sequence of graphs rather than as a set, but we prefer to describe $\Gcal$ as a set. Define an operation $\circ$ on $\Gcal$ to be \emph{admissible} if 
\begin{itemize}
\item $(G\circ G^{\pr})\circ G^{\pr\pr}=G\circ(G^{\pr}\circ G^{\pr\pr})$, and
\item for $G\in\Gcal_n,\ G^{\pr}\in\Gcal_m$:
  \begin{itemize}
  \item $G\circ G^{\pr}$ is in $\Gcal_{n+m}$,
  \item $G=(G\circ G^{\pr})|_{[n]}$, and
  \item $(G^{\pr}+n)=(G\circ G^{\pr})|_{[m]+n}$.
  \end{itemize}
\end{itemize}

If $\circ$ is admissible, we call the pair $(\Gcal,\circ)$ an \emph{admissible family}. We remark that our definition of admissibility is stronger than that of \cite[Definition 3.4]{ronco:2012tamari}, but this is the appropriate condition to define an associative algebra of maximal tubings; cf. \cite[Theorem 3.10]{ronco:2012tamari}.

Let $\MTub(\Gcal)$ be the set of all maximal tubings of these graphs:
$$\MTub(\Gcal)=\bigsqcup_{G\in\Gcal}\MTub(G)$$

We let $\Kbb[\Gcal]=\Kbb[\MTub(\Gcal)]$ be the $\Kbb$-vector space for which $\MTub(\Gcal)$ indexes a basis. We will consider the distinguished basis $\{\Pbb_{\Xcal}:\ \Xcal\in\MTub(\Gcal)\}$ for $\Kbb[\Gcal]$. The vector space $\Kbb[\Gcal]$ is graded so that an element $\Pbb_{\Xcal}$ is of degree $n$ if $\Xcal$ is a tubing of a graph $G$ with $n$ vertices. Since each $\Gcal_n$ is finite, each graded component of $\Kbb[\Gcal]$ is finite-dimensional.

\begin{definition}\label{def:tubing_mult}
  Let $G\in\Gcal_n$ and $G^{\pr}\in\Gcal_m$ be given. For maximal tubings $\Xcal\in\MTub(G)$ and $\Ycal\in\MTub(G^{\pr})$, define
  $$\Pbb_{\Xcal}\cdot\Pbb_{\Ycal}=\sum\Pbb_{\Zcal}$$
  where the sum is over all maximal tubings $\Zcal$ of $G\circ G^{\pr}$ such that $\Xcal=\Zcal|_{[n]}$ and $(\Ycal+n)=\Zcal|_{[m]+n}$.
\end{definition}

\begin{theorem}[Theorem 3.10 \cite{ronco:2012tamari}]\label{thm_admissible_associative}
  If $\Gcal$ is a family of graphs with an admissible operation $\circ$ as above, then the binary operation in Definition~\ref{def:tubing_mult} is associative.
\end{theorem}

To prove Theorem~\ref{thm_admissible_associative}, one may show directly that if $G\in\Gcal_n,\ G^{\pr}\in\Gcal_m,$ and $G^{\pr\pr}\in\Gcal_r$ are graphs with maximal tubings $\Xcal,\ \Ycal,$ and $\Zcal$, respectively, then
$$(\Pbb_{\Xcal}\Pbb_{\Ycal})\Pbb_{\Zcal}=\sum\Pbb_{\Wcal}=\Pbb_{\Xcal}(\Pbb_{\Ycal}\Pbb_{\Zcal})$$
where the sum is taken over $\Wcal\in\MTub(G\circ G^{\pr}\circ G^{\pr\pr})$ such that
\[ \Xcal=\Wcal|_{[n]},\ \Ycal+n=\Wcal|_{[m]+n},\ \Zcal+(n+m)=\Wcal|_{[r]+m+n}. \]

\begin{example}\label{ex_MR_alg}
  Consider the family of complete graphs $\Gcal=\{K_n\}_{n\geq 0}$ where we define $K_n\circ K_m=K_{n+m}$. If $\Xcal$ is any maximal tubing of $K_n$, its corresponding $K_n$-tree $\tau(\Xcal)$ is a chain. Letting $\Xcal\in\MTub(K_n),\ \Ycal\in\MTub(K_m)$, the elements $\Pbb_{\Zcal}$ in the support of $\Pbb_{\Xcal}\cdot\Pbb_{\Ycal}$ are indexed by precisely those tubings of $K_{n+m}$ for which $\tau(\Zcal)$ is a linear extension of $\tau(\Xcal)\sqcup \tau(\Ycal)$. But this is the shuffle product of $\tau(\Xcal)$ and $\tau(\Ycal)$ when viewed as permutations. Hence, the natural map $\Kbb[\Sfrak_{\infty}]\ra\Kbb[\Gcal]$ is an isomorphism of algebras from the Malvenuto-Reutenauer algebra to the tubing algebra on the family of complete graphs. A similar result about the coalgebra structure of $\Kbb[\Sfrak_{\infty}]$ will be given in Example~\ref{ex_MR_coalg}.
\end{example}

\begin{remark}\label{rem_MR_decorated}
  In many instances, it is useful to consider a generalization of the Malvenuto-Reutenauer algebra, which is indexed by \emph{decorated permutations}; see \cite{novelli2010free} or \cite{pilaud2018hopf}. A decorated permutation is a pair $(w,G)$ consisting of a permutation $w$ and an element $G$ called the decoration. If $\Gcal=\sqcup_{n=0}^{\infty}\Gcal_n$ is a graded set with an admissible operation $\circ$, one may define an algebra with a basis
  \[ \bigsqcup_{n=0}^{\infty}\{\Fbb_{(w,G)}:\ w\in\Sfrak_n,\ G\in\Gcal_n\} \]
  in much the same way as the Tubing algebra, where $\Fbb_{(u,G)}\cdot\Fbb_{(v,G^{\pr})}$ is the sum of $\Fbb_{(w,G\circ G^{\pr})}$ for which $w$ is a shuffle of $u$ and a shift of $v$. Likewise, the coalgebra structure can be extended to the decorated setting.

  For this paper, we have chosen to focus on the undecorated setting, though we expect many of our results to hold for decorated permutations as well.
\end{remark}

\begin{figure}
  
  \centering
  \includegraphics[scale=.8]{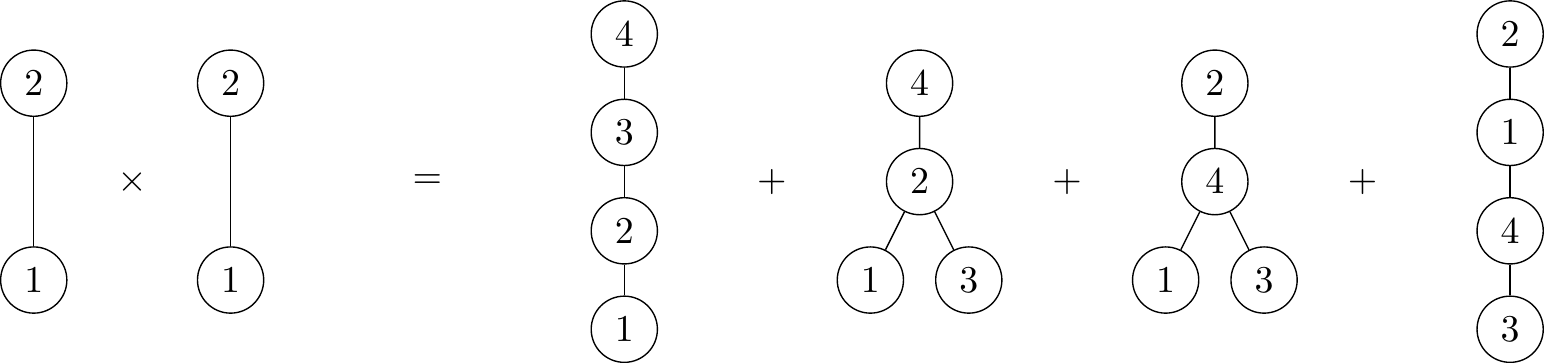}
  \caption{\label{fig:G2prod}This is the product of basis elements indexed by two $G$-trees (Example~\ref{ex_admissible}); for clarity, we remove the letter $\Pbb$ from this expression.}
  
\end{figure}

\begin{example}\label{ex_admissible}
  Let $G_n$ be the complete bipartite graph on $[n]$ where $i$ and $j$ are adjacent if $|i-j|$ is odd. It is straight-forward to check that $\{G_n\}_{n\geq 0}$ is an admissible family with $G_n\circ G_m=G_{n+m}$. The product of the basis elements indexed by the two $G$-trees for $G=G_2$ is shown in Figure~\ref{fig:G2prod}.

  Similarly, the sequences of path graphs, complete graphs, and edge-free graphs are admissible, so their tubings form the basis of an associative algebra. These algebras are the Loday-Ronco algebra, the Malvenuto-Reutenauer algebra, and the polynomial ring in one variable, respectively (c.f. \cite{forcey.springfield:2010geometric}).

  On the other hand, while the sequence of cycle graphs $C_n$ is not an admissible family, \cite{forcey.springfield:2010geometric} constructs a different binary operation to make the vector space $\Kbb[\Gcal]$ into an associative algebra. We leave the details of that construction to their paper.
\end{example}

For the remainder of this section, we make the assumption that $|\Gcal_n|=1$ for all $n\geq 0$. For clarity, we may refer to such a collection $\Gcal$ as a \emph{1-parameter family}. In this situation, there is at most one admissible operation $\circ$ defined by the fact that $G\circ G^{\pr}$ is in $\Gcal_{n+m}$ whenever $G\in\Gcal_n$ and $G^{\pr}\in\Gcal_m$. Hence, we simply say that the family $\Gcal$ is admissible if the operation $\circ$ is.

Our first main result in this section is a characterization of admissible families. For $A\subseteq\Nbb:=\{1,2,3,\ldots\}$, let $\Gcal(A)=\{G_n^A\}_{n\geq 0}$ be the family of graphs such that $V(G_n^A)=[n]$ and there is an edge between $i$ and $j$ if and only if $|j-i|\in A$.

\begin{proposition}\label{prop_admissible_characterization}
A 1-parameter family $\Gcal$ is admissible if and only if there exists $A\subseteq\Nbb$ such that $\Gcal=\Gcal(A)$.
\end{proposition}

\begin{proof}
  For a given $A\subseteq\Nbb$, it is clear that $\Gcal(A)$ is an admissible family. Indeed, it is clear from the definition that the restriction of $G_{n+m}^A$ to $[n]$ is equal to $G_n^A$, and the restriction of $G_{n+m}^A$ to $[m]+n$ is $G_m^A+n$, as desired.

  Now suppose $\Gcal=\{G_n\}_{n\geq 0}$ is an admissible family, and let $A=\{k\in\Nbb|\ (1,k+1)\in E(G_{k+1})\}$. We claim that $\Gcal=\Gcal(A)$. To this end, let $n\geq 1$ and $k\in A$ be given where $k\leq n-1$. Select $1\leq i<j\leq n$ such that $j-i=k$. We may decompose $G_n$ as $G_n=G_j\circ G_{n-j}$, so the edge $(i,j)$ is in $G_n$ if and only if it is in $G_j$. Furthermore, $G_j=G_{i-1}\circ G_{j-i+1}$, so $G_j$ has the edge $(i,j)$ exactly when $G_{j-i+1}+(i-1)$ does. By definition of $A$, this occurs exactly when $k\in A$. It follows that $G_n=G_n^A$.
\end{proof}

As a corollary, we may deduce the first part of Theorem~\ref{thm_main}.

\begin{proof}[Proof of Theorem~\ref{thm_main}(\ref{thm_main_1})]
  If $\Gcal$ is an admissible 1-parameter family of graphs, then $\Gcal=\Gcal(A)$ for some subset $A\subseteq\Nbb$. If each graph $G_n^A$ is filled, then for $i<j$, if $j\in A$ then $i\in A$. This is equivalent to the condition that there exists some $k\in\{0,1,2,\ldots\}\cup\{+\infty\}$ such that $A=\{i\in\Nbb:\ i\leq k\}$. But this means $G_n^A=H_{k,n}$ for all $n$. By Proposition~\ref{prop_translational_char}, this means that the sequence of lattice congruences $\mathbf{\Theta}=\{\Theta_n\}_{n\geq 0}$ corresponding to the filled graphs $\{G_n\}_{n\geq 0}$ form a translational family.
\end{proof}

Let $H,G$ be graphs on $[n]$ such that $E(H)\subseteq E(G)$. From the definition of the graph associahedron, the polytope $P_H$ is a Minkowski summand of $P_G$, so the normal fan of $P_H$ coarsens the normal fan of $P_G$. This in turn induces a surjective map $\Psi_H^G:\MTub(G)\ra\MTub(H)$.

For the remainder of the section, we fix subsets $A\subseteq B\subseteq\Nbb$. The graph $G_n^A$ is a subgraph of $G_n^B$ for all $n\geq 0$, which determines a surjective map $\MTub(G_n^B)\thra\MTub(G_n^A)$. For notational convenience, we write $\Psi_n$ in place of the map $\Psi_{G_n^A}^{G_n^B}$ for all $n\geq 0$.

\begin{lemma}\label{lem_psi_restriction}
  For $\Wcal\in\MTub(G_{n+m}^B)$:
  \begin{enumerate}
  \item\label{lem_psi_restriction_1} $\Psi_{n+m}(\Wcal)|_{[n]}=\Psi_n(\Wcal|_{[n]})$
  \item\label{lem_psi_restriction_2} $\std(\Psi_{n+m}(\Wcal)|_{[m]+n})=\Psi_m(\std(\Wcal|_{[m]+n}))$
  \end{enumerate}
\end{lemma}

\begin{proof}
  The standardization map in (\ref{lem_psi_restriction_2}) has the effect of shifting the vertex set from $[m]+n$ to $[m]$. Besides this point, the two parts are symmetric, so we only prove the first.

  Let $\Wcal$ be a maximal tubing of $G_{n+m}^B$ and set $\Zcal=\Psi_{n+m}(\Wcal)$. We wish to show that $\Zcal|_{[n]}$ is equal to $\Psi_n(\Wcal|_{[n]})$. Since they are both maximal tubings of $G_n$, it suffices to show that their $G$-trees share a common linear extension.

  Let $u=u_1\cdots u_{n+m}$ be a permutation of $[n+m]$ that is a linear extension of $\tau(\Wcal)$. Then $u$ is also a linear extension of $\tau(\Zcal)$, so $u|_{[n]}$ is a linear extension of $\tau(\Zcal)|_{[n]}$. On the other hand, $u|_{[n]}$ is a linear extension of $\tau(\Wcal|_{[n]})$, so it is also a linear extension of $\Psi_n(\Wcal|_{[n]})$.
\end{proof}

Now we return to the embedding $c:\Kbb[Z_{\infty}]\hookra\Kbb[\Sfrak_{\infty}]$ from Section~\ref{subsec_MR}. The maps $\{\Psi_n\}_{n\geq 0}$ give rise to an embedding of vector spaces $c:\Kbb[\Gcal(A)]\hookra\Kbb[\Gcal(B)]$ where
$$c(\Pbb_{\Xcal})=\sum_{\Ycal\in\Psi_n^{-1}(\Xcal)}\Pbb_{\Ycal}$$
for $\Xcal\in G_n^A$.

\begin{theorem}\label{thm_graph_MR_subalg}
  The embedding
  $$c:\Kbb[\Gcal(A)]\hookra\Kbb[\Gcal(B)]$$
  is a map of algebras.
\end{theorem}

\begin{proof}
  Let $\Xcal\in\MTub(G_n^A)$ and $\Ycal\in\MTub(G_m^A)$ be given. Then $c(\Pbb_{\Xcal}\cdot\Pbb_{\Ycal})=\sum c(\Pbb_{\Zcal})$, where the sum is over $\Zcal\in\MTub(G_{n+m}^A)$ such that $\Xcal=\Zcal|_{[n]}$ and $\Ycal+n=\Zcal|_{[m]+n}$. We have
  $$\sum_{\Zcal} c(\Pbb_{\Zcal})=\sum_{\Zcal}\sum_{\Wcal\in\Psi_{n+m}^{-1}(\Zcal)}\Pbb_{\Wcal}$$
  On the other hand,
  \begin{align*}
  c(\Pbb_{\Xcal})\cdot c(\Pbb_{\Ycal}) &=(\sum_{\Wcal^{\pr}\in\Psi_n^{-1}(\Xcal)}\Pbb_{\Wcal^{\pr}})\cdot (\sum_{\Wcal^{\pr\pr}\in\Psi_m^{-1}(\Ycal)}\Pbb_{\Wcal^{\pr\pr}})\\
  &=\sum_{\substack{\Wcal^{\pr}\in\Psi_n^{-1}(\Xcal)\\\Wcal^{\pr\pr}\in\Psi_m^{-1}(\Ycal)}}\Pbb_{\Wcal^{\pr}}\cdot\Pbb_{\Wcal^{\pr\pr}}
  \end{align*}

  We show that $c(\Pbb_{\Xcal}\cdot \Pbb_{\Ycal})=c(\Pbb_{\Xcal})\cdot c(\Pbb_{\Ycal})$. To this end, fix $\Pbb_{\Zcal}$ in the expansion of $\Pbb_{\Xcal}\cdot \Pbb_{\Ycal}$, and let $\Wcal\in\Psi_{n+m}^{-1}(\Zcal)$. Set $\Wcal^{\pr}=\Wcal|_{[n]}$ and $\Wcal^{\pr\pr}+n=\Wcal_{[m]+n}$ so that $\Wcal^{\pr}\in\MTub(G_n^B)$ and $\Wcal^{\pr\pr}\in\MTub(G_m^B)$. Clearly, $\Pbb_{\Wcal}$ is in the expansion of $\Pbb_{\Wcal^{\pr}}\cdot \Pbb_{\Wcal^{\pr\pr}}$. But,
  \begin{align*}
    &\Psi_n(\Wcal^{\pr}) =\Psi_{n+m}(\Wcal)|_{[n]}=\Zcal|_{[n]}=\Xcal,\ \hspace{2mm}\mbox{and}\\
    &\Psi_m(\Wcal^{\pr\pr})=\std(\Psi_{n+m}(\Wcal)|_{[m]+n})=\std(\Zcal|_{[m]+n})=\Ycal,
  \end{align*}
  so $\Pbb_{\Wcal}$ is in the expansion of $c(\Pbb_{\Xcal})\cdot c(\Pbb_{\Ycal})$.
  Conversely, suppose $\Wcal^{\pr}\in\Psi_n^{-1}(\Xcal)$ and $\Wcal^{\pr\pr}\in\Psi_m^{-1}(\Ycal)$ are given, and let $\Pbb_{\Wcal}$ an element in the expansion of $\Pbb_{\Wcal^{\pr}}\cdot\Pbb_{\Wcal^{\pr\pr}}$. Set $\Zcal=\Psi_{n+m}(\Wcal)$. Then
  \begin{align*}
    &\Zcal|_{[n]}=\Psi_{n+m}(\Wcal)|_{[n]}=\Psi_n(\Wcal^{\pr})=\Xcal,\ \hspace{2mm}\mbox{and}\\
    &\std(\Zcal|_{[m]+n})=\std(\Psi_{n+m}(\Wcal)|_{[m]+n}=\Psi_m(\Wcal^{\pr\pr})=\Ycal,
  \end{align*}
  so $\Pbb_{\Zcal}$ is in the expansion of $\Pbb_{\Xcal}\cdot \Pbb_{\Ycal}$. Both $c(\Pbb_{\Xcal}\cdot \Pbb_{\Ycal})$ and $c(\Pbb_{\Xcal})\cdot c(\Pbb_{\Ycal})$ are multiplicity-free sums of basis elements with the same support, so they are equal.
\end{proof}

\begin{corollary}
  If $\Gcal$ is an admissible 1-parameter family, the tubing algebra $\Kbb[\Gcal]$ is a subalgebra of the Malvenuto-Reutenauer algebra.
\end{corollary}

\subsection{Tubing coalgebras}\label{subsec_tubing_coalgebra}

We next define a comultiplication on $\Kbb[\Gcal]$. We will assume throughout that $\Gcal$ is a 1-parameter family, though it should be possible to extend it to more general families of graphs by defining a ``selection'' operation as in \cite{pilaud2018hopf}.

Say that $\Gcal$ is \emph{restriction-compatible} if for any $G\in\Gcal$ and any subset of vertices $I\subseteq V(G)$,

\begin{itemize}
\item $\std(G|_I)$ is a subgraph of the graph $G^{\pr}\in\Gcal$ where $V(G^{\pr})=V(\std(G|_I))$, and
\item $\std(G/I)$ is a subgraph of the graph $G^{\pr\pr}\in\Gcal$ where $V(G^{\pr\pr})=V(\std(G/I))$.
\end{itemize}

We note that the second property actually implies the first since $\std(G|_I)$ is a subgraph of $\std(G/(V\setm I))$.

\begin{example}\label{ex_res_comp}
  Path graphs, complete graphs, and edge-free graphs are all restriction-compatible in addition to being admissible (Example~\ref{ex_admissible}). In these cases, the quotient graphs $G/I,\ I\subseteq V(G)$ are again path graphs, complete graphs, and edge-free graphs, respectively. Similarly, the sequence of cycle graphs $C_n$ whose vertices are labeled in cyclic order are also restriction-compatible since the quotient graphs are all cycles. On the other hand, the family of complete bipartite graphs in Example~\ref{ex_admissible} is not restriction-compatible.
\end{example}

We will not attempt to completely describe all restriction-compatible families of graphs, but we may describe those families that are both restriction-compatible and admissible.

\begin{proposition}
  If $\Gcal$ is a 1-parameter family of graphs that is both restriction-compatible and admissible, then $\Gcal$ must be either the set of path graphs, complete graphs, or edge-free graphs.
\end{proposition}

\begin{proof}
  To be an admissible family, $\Gcal$ must be equal to $\Gcal(A)$ for some set $A\subseteq\Nbb$. We wish to show that restriction-compatibility forces either $A=\{1\},\ A=\Nbb$, or $A=\emptyset$. Restriction-compatibility of these cases was observed in Example~\ref{ex_res_comp}. To prove that these are the only examples, it is enough to show that if there exists $k\in A,\ k\geq 2$, then $A=\Nbb$.

  Suppose such $k$ exists, and let $j\in\Nbb$ with $j<k$. Let $H=(G_{k+1})|_{[j]\cup\{k+1\}}$. Since $\{1,k+1\}\in E(G_{k+1})$, the graph $\std(H)$ is a subgraph of $G_{j+1}$ containing the edge $\{1,j+1\}$. This implies $j\in A$.

  On the other hand, suppose $j\in\Nbb$ with $j>k$ and set $n=(j+1)k$. Let $I\subseteq[n]$ such that
  \begin{enumerate}
  \item $|I|=j-1$,
  \item $I$ does not contain any multiples of $k$, and
  \item the smallest element of $I$ is greater than $k$.
  \end{enumerate}

Such a collection exists since $k\geq 2$. Now let $J=[n]\setm (I\cup\{k,n\})$. Then the graph $\std(G_n/J)$ is a subgraph of $G_{j+1}$ containing the edge $\{1,j+1\}$. Hence, $j\in A$ holds.
\end{proof}

If $H$ is a subgraph of $G$ with the same vertex set $[n]$ and $\Xcal$ is in $\MTub(H)$, we let $c_H^G(\Xcal)=\sum\Ycal$ where the sum ranges over $\Ycal\in\MTub(G)$ such that $\Psi_H^G(\Ycal)=\Xcal$. Suppose $\Gcal$ is a restriction-compatible family. We define
$$\Delta_{\Gcal}=\Delta:\Kbb[\Gcal]\ra\Kbb[\Gcal]\otimes\Kbb[\Gcal]$$
as follows. If $\Xcal\in\MTub(G)$, let

$$\Delta(\Pbb_{\Xcal})=\sum c_{\std(G|_I)}^{G^{\pr}}(\Pbb_{\std(\Xcal|_I)})\otimes c_{\std(G/I)}^{G^{\pr\pr}}(\Pbb_{\std(\Xcal/I)}),$$

where the sum is over ideals $I$ of $\Xcal$, and $G^{\pr},G^{\pr\pr}\in\Gcal$ such that $|I|=|V(G^{\pr})|$ and $|V(G)\setm I|=|V(G^{\pr\pr})|$.

\begin{figure}
  
  \centering
  \includegraphics[scale=.75]{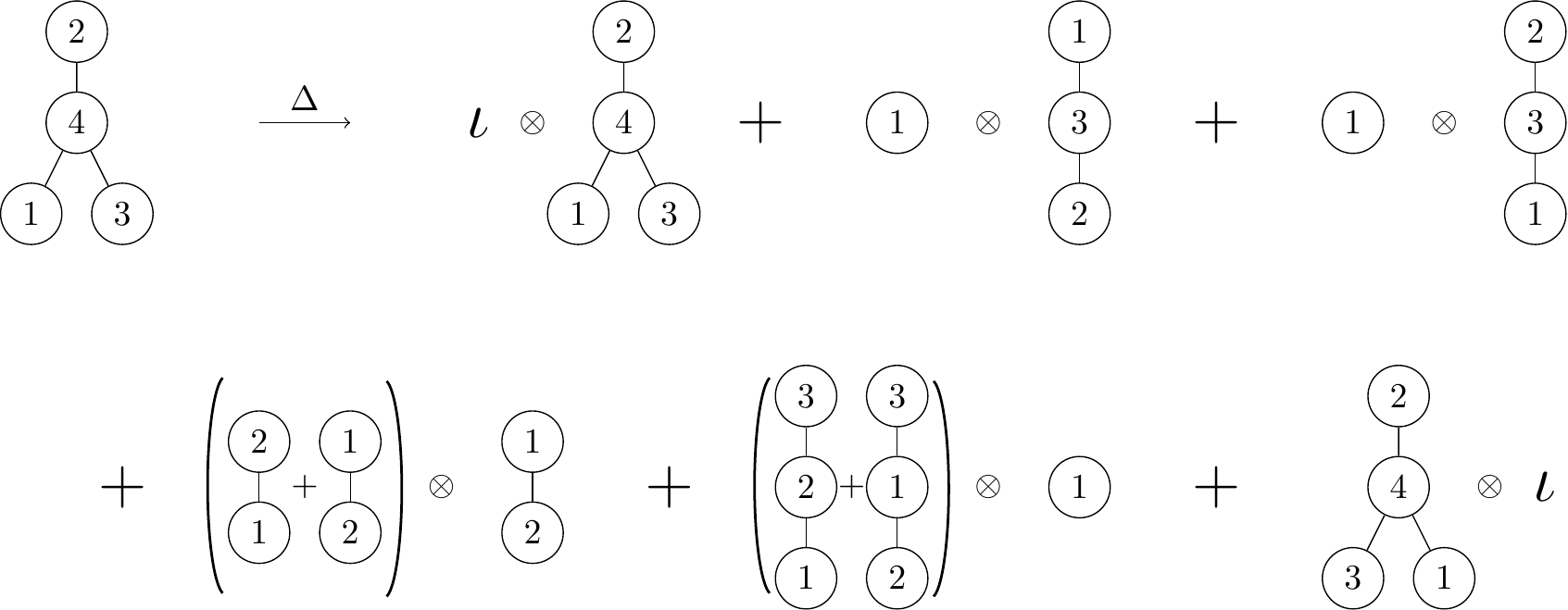}
  \caption{\label{fig:coprod}This is the comultiplication of a basis element indexed by a $G$-tree for the family of cycle graphs (Example~\ref{ex_cycle_coalg}); for clarity, we remove the letter $\Pbb$ from this expression.}
  
\end{figure}

\begin{example}\label{ex_MR_coalg}
  We again consider the case $\Gcal=\{K_n\}_{n\geq 0}$ from Example~\ref{ex_MR_alg}. Every induced subgraph $H$ of $K_n$ is a complete graph, as is the quotient $K_n/H$. Thus, for $\Xcal\in\MTub(K_n)$, the formula for $\Delta$ simplifies to
  $$\Delta(\Pbb_{\Xcal})=\sum\Pbb_{\std(\Xcal|_{I})}\otimes \Pbb_{\std(\Xcal/I)},$$
  where the sum ranges over the ideals of $\Xcal$. Since $\tau(\Xcal)$ is a chain $u_1<\cdots<u_n$, its order ideals are of the form $\{u_1,\ldots,u_i\}$ for $i=0,1,\ldots,n$. Under the bijection between $\MTub(K_n)$ and $\Sfrak_n$, this expression becomes
  $$\Delta(\Fbb_u)=\sum\Fbb_{\std(u_1\cdots u_i)}\otimes \Fbb_{\std(u_{i+1}\cdots u_n)}.$$
  Thus, $\Kbb[\Gcal]$ has the same coalgebra structure as $\Kbb[\Sfrak_{\infty}]$.
\end{example}

\begin{example}\label{ex_cycle_coalg}
  The set $\{C_n\}_{n\geq 0}$ of cyclically ordered cycle graphs is another restriction-compatible family. In Figure~\ref{fig:coprod} we show the comultiplication applied to a $C_4$-tree, or equivalently, a maximal tubing $\Xcal$ of $C_4$. The sum is split into six terms, one for each choice of ideal of $\Xcal$. We observe that for the two ideals $I$ such that $G|_I$ is not a cycle graph, the element $c_{\std(G|_I)}^{G^{\pr}}(\Pbb_{\std(\Xcal_I)})$ has multiple summands.

For example, the fourth term corresponds with the ideal $I=\{1,3\}$.
Observe that $\std(G|_I)$ is the edge-free graph on $[2]$, $\Xcal|_{\std(\{1,3\})}=\{\{1\}, \{2\}\}$, and the corresponding $G$-forest $T$ has $1$ and $2$ incomparable.
Since $C_2$ is also the complete graph on $[2]$, each element of $\Psi_{G}^{H}$ fiber of $\Xcal|_{\std(\{1,3\})}$ is just a linear extension of $T$.
\end{example}

\begin{theorem}
  If $\Gcal$ is a restriction compatible family, then the map
  $$c:\Kbb[\Gcal]\hookra\Kbb[\Sfrak_{\infty}]$$
  commutes with $\Delta$. In particular, $\Delta_{\Gcal}$ is coassociative.
\end{theorem}

\begin{proof}
  Fix a maximal tubing $\Xcal\in\MTub(G_n)$. We show that $\Delta(c(\Pbb_{\Xcal}))=(c\otimes c)\circ(\Delta_{\Gcal}(\Pbb_{\Xcal}))$.

  The element $c(\Pbb_{\Xcal})$ is supported by the permutations of $[n]$ that are linear extensions of the tree poset $\tau(\Xcal)$. Let $\Lcal(P)$ be the set of linear extensions of a poset $P$. Then,

  \begin{align*}
    \Delta(c(\Pbb_{\Xcal})) &= \sum_{u\in\Lcal(\tau(\Xcal))}\Delta(\Fbb_u)\\
    &= \sum_{i=0}^n\sum_{u\in\Lcal(\tau(\Xcal))}\Fbb_{\std(u_1\cdots u_i)}\otimes \Fbb_{\std(u_{i+1}\cdots u_n)}
  \end{align*}

  If $u=u_1\cdots u_n$ is a linear extension of $\tau(\Xcal)$, then the subset $\{u_1,\ldots,u_i\}$ is an ideal, and the complement $\{u_{i+1},\ldots,u_n\}$ is an order filter. If $I$ is an order ideal, then $\tau(\Xcal)|_I=\tau(\Xcal|_I)$ and $\tau(\Xcal)|_{[n]\setm I}=\tau(\Xcal/I)$. Putting these together, we have

  \begin{align*}
    \Delta(c(\Pbb_{\Xcal})) &= \sum_I(\sum_{u\in\Lcal(\tau(\Xcal)|_I)} \Fbb_{\std(u)}) \otimes (\sum_{w\in\Lcal(\tau(\Xcal)|_{[n]\setm I})}\Fbb_{\std(w)})\\
    &= \sum_I(\sum_{u\in\Lcal(\tau(\Xcal|_I))} \Fbb_{\std(u)}) \otimes (\sum_{w\in\Lcal(\tau(\Xcal/I))}\Fbb_{\std(w)})\\
    &= \sum_I c(\Pbb_{\std(\Xcal|_I)})\otimes c(\Pbb_{\std(\Xcal/I)}),
  \end{align*}

  where the sum ranges over ideals $I$ of $\tau(\Xcal)$. If $K\subseteq H\subseteq G$ is a sequence of subgraphs with a common vertex set $[n]$, then the map $c_K^G$ factors as $c_K^G=c_H^G\circ c_K^H$. Since $\Gcal$ is a restriction-compatible family,

   $$\sum_I c(\Pbb_{\std(\Xcal|_I)}) \otimes c(\Pbb_{\std(\Xcal/I)}) = \sum_I c_{G^{\pr}}^{K_{|I|}}c_{\std(G|_I)}^{G^{\pr}}(\Pbb_{\std(\Xcal|_I)})\otimes c_{G^{\pr\pr}}^{K_{|V\setm I|}}c_{\std(G/I)}^{G^{\pr\pr}}(\Pbb_{\std(\Xcal/I)}),$$

  where $G^{\pr},G^{\pr\pr}\in\Gcal$ such that $|V(G^{\pr})|=|I|$ and $|V(G^{\pr\pr})|=|V\setm I|$. The latter sum simplifies to

  $$(c\otimes c)\left(\sum_I c_{\std(G|_I)}^{G^{\pr}}(\Pbb_{\std(\Xcal|_I)})\otimes c_{\std(G/I)}^{G^{\pr\pr}}(\Pbb_{\std(\Xcal/I)})\right)=(c\otimes c)\circ(\Delta_{\Gcal}(\Pbb_{\Xcal})),$$
  as desired.
\end{proof}

\subsection{Insertional families of lattice congruences}\label{subsec_insertional}

If $\alpha=(i,j,\epsilon)$ is an arc on $[n]$, we define the \emph{deletion} $\alpha\setm k$ to be the arc on $[n-1]$ where
$$\alpha\setm k=\begin{cases}(i-1,j-1,\epsilon)\ \mbox{if }k<i\\(i,j,\epsilon)\ \mbox{if }k>j\\(i,j-1,\epsilon^{\pr})\ \mbox{if }i\leq k\leq j\end{cases},$$
where $\epsilon^{\pr}_l=\epsilon_l$ when $l\leq k-i$ and $\epsilon^{\pr}_l=\epsilon_{l+1}$ when $l>k-i$. That is, $\epsilon^{\pr}$ is obtained from $\epsilon$ by deleting some $+$ or $-$ entry. Reversing this operation, we say the arc $\beta$ is obtained from $\alpha$ by \emph{inserting} $k$ if $\alpha=\beta\setm k$.

A sequence of lattice congruences $\mathbf{\Theta}=\{\Theta_n\}_{n\geq 0}$ is an \emph{insertional family} if for any arc $\alpha$ contracted by $\Theta_n$, any arc $\beta$ obtained by inserting some $k\in[n+1]$ is contracted by $\Theta_{n+1}$. The analogue of Theorem~\ref{thm_translational_subalg} proved in \cite[Theorem 1.3, Proposition 8.1]{reading:2005lattice} is as follows.

\begin{theorem}
  If $\mathbf{\Theta}=\{\Theta_n\}_{n\geq 0}$ is an insertional family, then the map
  $$c:\Kbb[Z_{\infty}^{\mathbf{\Theta}}]\ra\Kbb[\Sfrak_{\infty}]$$
  embeds $\Kbb[Z_{\infty}^{\mathbf{\Theta}}]$ as a sub-coalgebra of $\Kbb[\Sfrak_{\infty}]$.
\end{theorem}

We now prove the second part of Theorem~\ref{thm_main}.

\begin{proof}[Proof of Theorem~\ref{thm_main}(\ref{thm_main_2})]
Let $\Gcal=\{G_n\}_{n\geq 0}$ be a 1-parameter family of filled graphs, and let $\mathbf{\Theta}=\{\Theta_n\}_{n\geq 0}$ be the corresponding sequence of lattice congruences. We must prove that $\Gcal$ is restriction-compatible if and only if $\mathbf{\Theta}$ is insertional.

Suppose first that $\mathbf{\Theta}$ is an insertional family of lattice congruences. To prove that $\Gcal$ is restriction-compatible, it suffices to show that $\std(G_n/\{i_1,\ldots,i_l\})$ is a subgraph of $G_{n-l}$ for $1\leq i_1<\cdots<i_l\leq n$. Indeed, it is enough to prove this statement for $l=1$ since if $H$ and $G$ are graphs on $[n]$ such that $E(H)\subseteq E(G)$, the quotient $H/i$ is a subgraph of $G/i$ for any $i\in[n]$. Hence, the statement for $l=1$ gives a sequence of inclusions:
$$E(\std(G_n/\{i_1,\ldots,i_l\}))\subseteq\cdots\subseteq E(\std(G_{n-l+1}/\{i_1\}))\subseteq E(G_{n-l}).$$

For $k\in[n]$, we show that $\std(G_n/k)$ is a subgraph of $G_{n-1}$. Suppose $\{i,j\}$ is not an edge of $G_{n-1}$. Then $\Theta_{n-1}$ contracts the arc $\alpha=(i,j,+)$. Let $\beta^+=(i^{\pr},j^{\pr},+)$ be the arc obtained from $\alpha$ by inserting $k$ such that its sign vector $\epsilon$ is $(+,\ldots,+)$. If $i<k<j+1$, then there is another arc $\beta^-=(i^{\pr},j^{\pr},\epsilon^{\pr})$ such that $\epsilon^{\pr}_{k-i^{\pr}}=-$. Since $\mathbf{\Theta}$ is insertional, both $\beta^+$ and $\beta^-$ are contracted by $\Theta_n$.

We claim that $\{i,j\}$ is not an edge of $\std(G_n/k)$. If, to the contrary, it is an edge of $\std(G_n/k)$, then either $\{i^{\pr},j^{\pr}\}$ is an edge of $G_n$, or $\{i^{\pr},k\}$ and $\{k,j^{\pr}\}$ are both edges of $G_n$. In the former case, the arc $\beta^+$ is not contracted by $\Theta_n$, a contradiction. On the other hand, suppose $\{i^{\pr},j^{\pr}\}$ is not an edge, but $\{i^{\pr},k\}$ and $\{k,j^{\pr}\}$ both are. Then $i^{\pr}<k<j^{\pr}$ holds since $G_n$ is filled. But this means $i^{\pr}=i$ and $j^{\pr}=j+1$, so the arc $\beta^-$ is well-defined, and it is contracted by $\Theta_n$. Since $\Theta_n$ is generated by positive arcs, either $(i^{\pr},k,+)$ or $(k,j^{\pr},+)$ must be contracted by $\Theta_n$. But this contradicts the assumption that $\{i^{\pr},k\}$ and $\{k,j^{\pr}\}$ are edges of $G_n$.

Now assume that $\Gcal$ is a restriction-compatible family. Let $\alpha=(i,j,\epsilon)$ be an arc contracted by $\Theta_{n-1}$, and pick $k\in[n]$. We claim that any arc $\beta$ obtained by inserting $k$ into $\alpha$ is contracted by $\Theta_n$. This will prove that $\mathbf{\Theta}$ is an insertional family.

Since $\Theta_{n-1}$ is generated by positive arcs, there exists a positive subarc $\alpha^{\pr}=(i^{\pr},j^{\pr},+)$ of $\alpha$ that is contracted by $\Theta_{n-1}$. As a result, the pair $\{i^{\pr},j^{\pr}\}$ is not an edge of $G_{n-1}$. Moreover, any arc $\beta$ of $[n]$ with $\beta\setm k=\alpha$ contains a subarc $\beta^{\pr}$ such that $\beta^{\pr}\setm k=\alpha^{\pr}$. Hence, to show that $\beta$ is contracted by $\Theta_n$, it is enough to show that $\beta^{\pr}$ is contracted by $\Theta_n$.

If $\beta^{\pr}$ is a positive arc, then it follows that $\beta^{\pr}$ is contracted by $\Theta_n$ since $E(\std(G_n\setm k))\subseteq E(G_{n-1})$. If $\beta^{\pr}$ is not a positive arc, then $i^{\pr}<k\leq j^{\pr}$ and $\beta^{\pr}=(i^{\pr},j^{\pr}+1,\epsilon^{\pr})$ where $\epsilon^{\pr}_{k-i^{\pr}}$ is the only negative entry in $\epsilon^{\pr}$. In this case, since $E(\std(G_n/k))\subseteq E(G_{n-1})$, either $\{i^{\pr},k\}$ or $\{k,j^{\pr}+1\}$ is not an edge, which means that some subarc of $\beta^{\pr}$ is contracted by $\Theta_n$. It follows that $\beta^{\pr}$ is contracted by $\Theta_n$ as well.
\end{proof}

\section{Open problems}\label{sec:other}

\subsection{Lattices of maximal tubings}\label{subsec:tubing_lattice}

Not every poset of maximal tubings is a lattice. For example, the two indicated atoms of the poset of maximal tubings shown in Figure~\ref{fig_nl} has two minimal upper bounds, so it is not a lattice.

\begin{figure}
  
  \centering
  \includegraphics{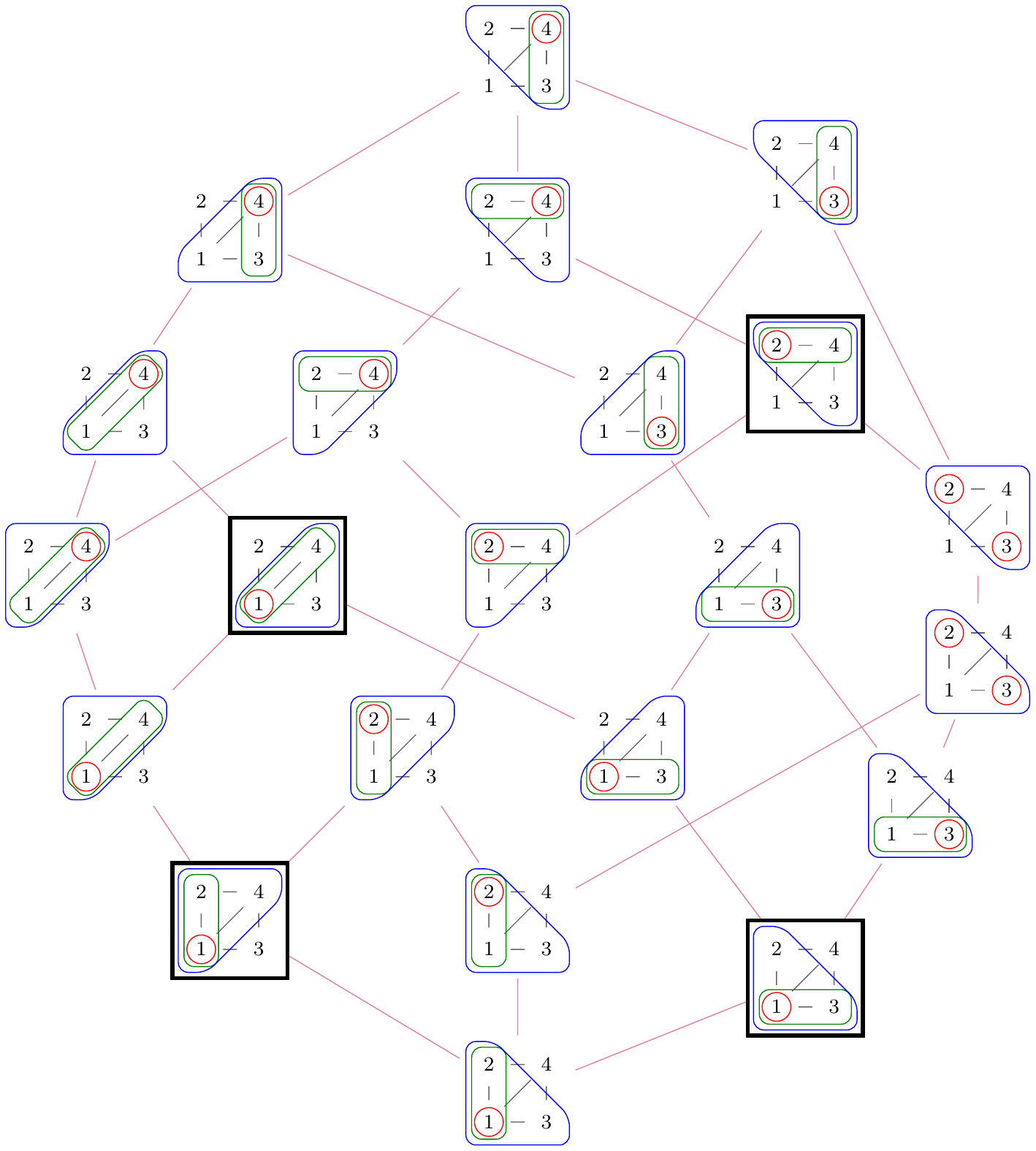}
  \caption{\label{fig_nl}A poset of maximal tubings that is not a lattice}
  
\end{figure}

Corollary~\ref{g-perm lattice cong} characterizes graphs $G$ for which $L_G$ is a meet-semilattice quotient of the weak order. A more fundamental problem is to characterize all graphs such that $L_G$ is a lattice. To this end, we make the simple observation that an interval $L^{\pr}$ of a lattice $L$ is a sublattice of $L$. In particular if $G^{\pr}$ is any graph obtained by contracting or deleting vertices of $G$ such that $L_{\std(G^{\pr})}$ is not a lattice, then $L_G$ is not a lattice either. Continuing to borrow from matroid terminology, we say that $G^{\pr}$ is a \emph{minor} of $G$ if it is the standardization of a sequence of contractions and deletions.

\begin{problem}
  Give an explicit list of minors such that $L_G$ is a lattice whenever $G$ does not contain a minor from the list.
\end{problem}

By exhaustive search, we found that when $G$ is a connected graph with four vertices, the poset $L_G$ is not a lattice if and only if $\{1,3\}$ and $\{2,4\}$ are edges but $\{2,3\}$ is not an edge in $G$. These are the seven graphs shown in Figure~\ref{fig_nlex}.

\begin{figure}
  \centering
  \includegraphics{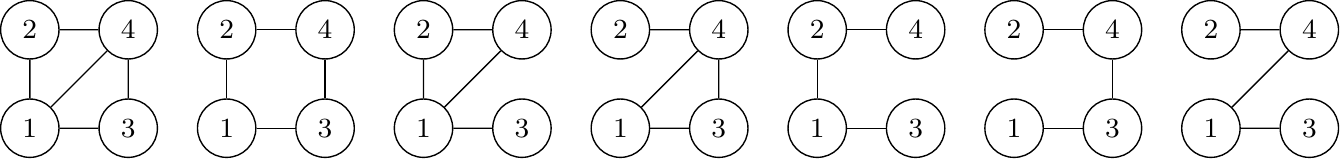}
  \caption{\label{fig_nlex}Graphs with four vertices such that $L_G$ is not a lattice}
\end{figure}

\subsection{Cyclohedra}\label{subsec:cycles}

Let $C_n$ be the $n$-cycle graph, with vertices labeled $1,2,\ldots,n$ in cyclic order. The graph associahedron $P_{C_n}$ is known as a \emph{cyclohedron}. The cyclohedron is combinatorially equivalent to the \emph{Type $B_{n-1}$ associahedron} \cite{simion:2003typeB}. Its facial structure is usually described in terms of Type $B_n$ Coxeter-Catalan combinatorial objects, e.g. centrally symmetric triangulations of polygons. The graph associahedron $P_{C_n}$ does not have the same normal fan as the Type $B_{n-1}$ associahedron, however. This geometric distinction is relevant in many of its applications. The graph associahedron $P_{C_n}$ is used to study the self-linking of knots \cite{bott.taubes:1994self} or to tile the moduli space $\ov{Z}^n$ in \cite{devadoss:2002space}, whereas the Type $B_n$ associahedron arises in the theory of cluster algebras \cite{fomin.zelevinsky:2003clusterII}.

From the Coxeter-Catalan point of view, the vertices of the Type $B_n$ associahedron can be partially ordered in several ways, which are called Cambrian lattices; \cite{reading:2006cambrian},\cite{thomas:2006tamari}. A \emph{Cambrian lattice} is a certain lattice quotient of the weak order of a finite Coxeter system. We remark that the poset of maximal tubings $L_{C_n}$ is not isomorphic to a Type $B_{n-1}$ Cambrian lattice for $n\geq 3$, despite the fact that they arise as orientations of the same undirected graph. Indeed, $L_{C_3}=L_{K_3}$ is the weak order of Type $A_2$, which is not isomorphic to any Cambrian lattice of Type $B_2$.

Cambrian lattices have a remarkable structure: they are all semidistributive lattices \cite{reading:2006cambrian}. A lattice is \emph{semidistributive} if for any three elements $x,y,z$: 
\begin{itemize}
\item if $x\wedge z=y\wedge z$, then $(x\vee y)\wedge z=x\wedge z$ and
\item if $x\vee z=y\vee z$, then $(x\wedge y)\vee z=x\vee z$.
\end{itemize}

The weak order is known to be semidistributive, so when $G$ is filled, the poset $L_G$ inherits semidistributivity as a lattice quotient of the weak order. We do not know of a way to represent $L_{C_n}$ as a lattice quotient of the weak order for $n\geq 4$. In particular, the canonical map $\Psi_{C_n}:\Sfrak_n\ra L_{C_n}$ is not a lattice map as $C_n$ is not filled for $n\geq 4$. However, we have verified by computer calculation that $L_{C_n}$ is a semidistributive lattice for $n\leq 6$. This has led us to the following question.

\begin{question}
  Is $L_{C_n}$ a semidistributive lattice for each $n\geq 1$?
\end{question}

We remark that the poset $L_G$ need not be semidistributive even when it is a lattice. For example, on may check that the star graph $G$ with $E(G)=\{\{1,2\},\{1,3\},\{1,4\}\}$ has a lattice of maximal tubings that is not semidistributive.




\subsection{Facial weak order}\label{subsec:facial_weak}

For $n\geq 0$, let $\Pi_n$ be the set of \emph{ordered set partitions} $(B_1,\ldots,B_l)$ of $[n]$. In \cite{chapoton:2000algebres}, Chapoton defined a Hopf algebra $\Kbb[\Pi_{\infty}]=\bigoplus\Kbb[\Pi_n]$ on the set of ordered set partitions. Identifying maximally refined ordered set partitions $(B_1,\ldots,B_n)$ with permutations, the natural inclusion $\Kbb[\Sfrak_{\infty}]\ra\Kbb[\Pi_{\infty}]$ is a Hopf algebra map.

This led to the development of the \emph{facial weak order} by Palacios and Ronco \cite{palacios:2006weak}, which is a partial ordering on $\Pi_n$ distinct from the usual refinement order. Under this poset, the product of two ordered set partitions is a sum of elements in an interval of the facial weak order. Dermenjian, Hohlweg, and Pilaud \cite{dermenjian:2018facial} proved that the facial weak order on $\Pi_n$ is a lattice for all $n\geq 1$. Furthermore, they show that any lattice congruence of the weak order may be ``lifted'' to a lattice congruence of the facial weak order. This suggests the following question:


\begin{question}
  Does a translational (resp. insertional) family $\mathbf{\Theta}=\{\Theta_n\}_{n\geq 0}$ of lattice congruences of the weak order lift to a family $\hat{\mathbf{\Theta}}=\{\hat{\Theta}_n\}_{n\geq 0}$ of congruences of the facial weak order such that $\Kbb[\Pi_{\infty}/\hat{\mathbf{\Theta}}]$ is a subalgebra (resp. sub-coalgebra) of $\Kbb[\Pi_{\infty}]$?
\end{question}









\section*{Acknowledgements}

The second author was supported by NSF/DMS-1440140 while in residence at the Mathematical Sciences Research Institute in Fall 2017.
We thank Vincent Pilaud and Ricky Ini Liu for helpful suggestions.

\bibliography{bib_graph_assoc}{}
\bibliographystyle{plain}

\end{document}